\documentclass{article}
\usepackage{graphicx,amssymb,amsmath,amsthm,amsopn,mathabx,fullpage,color}
\DeclareMathOperator{\spn}{span}
\DeclareMathOperator{\supp}{supp}
\DeclareMathOperator{\cone}{cone}
\DeclareMathOperator{\cube}{Z}

\newcommand{\cS}{\mathcal{S}}
\newcommand{\bH}{\mathbf{H}}
\newcommand{\bt}{\mathbf{t}}
\newcommand{\R}{\mathbb{R}}
\newcommand{\N}{\mathbb{N}}
\newcommand{\beq}{\begin{equation}}
\newcommand{\eeq}{\end{equation}}
\newcommand{\inn}[2]{\left\langle #1, #2\right\rangle}
\newcommand{\scz}[1]{\scriptstyle{#1}}
\newcommand{\diff}[1]{\,\mathrm{d}#1}
\newcommand{\vb}{\, \bigr| \,}
\newcommand{\bx}{\mathbf{x}}
\newcommand{\by}{\mathbf{y}}
\newcommand{\be}{\mathbf{e}}
\newcommand{\bX}{\mathbf{X}}
\newcommand{\bXm}{\mathbf{X}^-}
\newcommand{\bE}{\mathbf{E}}
\newcommand{\bxd}{\tilde{\mathbf{x}}}

\newcommand{\bomega}{\boldsymbol{\omega}}
\newcommand{\balpha}{\boldsymbol{\alpha}}
\newcommand{\Cspline}{C(\bx \vb \bX)}
\newcommand{\Bspline}{B(\bx \vb \bX)}

\newcommand{\Cnspline}{C_\mathbf{n}(\bx \vb \bX)}
\newcommand{\CnsplineE}{C_\mathbf{n}(\bx \vb \bE)}
\newcommand{\Bnspline}{B_\mathbf{n}(\bx \vb \bX)}
\newcommand{\Bnsplinem}{B_\mathbf{n}(\bx \vb \bX^-)}
\newcommand{\Calpha}{C_{\boldsymbol{\alpha}}}
\newcommand{\Calphaspline}{C_{\boldsymbol{\alpha}}(\bx \vb \bX)}
\newcommand{\CalphasplineE}{C_{\boldsymbol{\alpha}}(\bx \vb \bE)}
\newcommand{\Balpha}{B_{\boldsymbol{\alpha}}}

\newcommand{\Balphasplinem}{B_{\boldsymbol{\alpha}}(\bx \vb \bXm)}

\newcommand{\Balphasplineb}{B_{\boldsymbol{\alpha}}(\bullet \vb \bXm)}
\newcommand{\norm}[1]{\lVert #1 \rVert}
\newcommand{\bmu}{{\boldsymbol{\mu}}}
\newcommand{\bk}{{\mathbf{k}}}
\newcommand{\bz}{{\mathbf{z}}}
\newcommand{\Z}{\mathbb{Z}}
\newcommand{\Q}{\mathbb{Q}}
\DeclareMathOperator{\re}{Re}
\DeclareMathOperator{\im}{Im}
\DeclareMathOperator{\sinc}{sinc}
\DeclareMathOperator{\card}{card}
\newcommand{\bm}{{\mathbf{m}}}

\newtheorem{conedefn}{Definition}[section]
\newtheorem{generalizedremark}{Remark}[section]
\newtheorem{boxdefn}[conedefn]{Definition}
\newtheorem{conesdirection}{Proposition}[section]
\newtheorem{iteratedintegral}{Lemma}[section]
\newtheorem{cone3direction}[conesdirection]{Proposition}
\newtheorem{transformationremark}{Remark}[section]
\newtheorem{condatremark}[transformationremark]{Remark}
\newtheorem{generalcoeffs}{Proposition}[section]
\newtheorem{fracconedefn}{Definition}[section]
\newtheorem{fraccones}{Proposition}[section]
\newtheorem{fracconerecurrence}[fraccones]{Proposition}
\newtheorem{propfraccone3}[fraccones]{Proposition}
\newtheorem{reductionremark}{Remark}[section]
\newtheorem{frachexdefn}[fracconedefn]{Definition}
\newtheorem{propfrachex3}[fraccones]{Proposition}
\newtheorem{differenceremark}[reductionremark]{Remark}
\newtheorem{rieszbasis}{Theorem}[section]
\newtheorem{complexconedefn}{Definition}[section]
\newtheorem{complexhexdefn}[complexconedefn]{Definition}

\title{Fractional Cone Splines and Hex Splines}

\author{Peter R. Massopust and Patrick J. Van Fleet}

\begin{document}

\maketitle

\bibliographystyle{plain}
\begin{abstract}
    We introduce an extension of cone splines and box splines to fractional and complex orders. These new families of multivariate splines are defined in the Fourier domain along certain $s$-dimensional meshes and include as special cases the three-directional box splines \cite{article:condat} and hex splines \cite{article:vandeville} previously considered by Condat, Van De Ville et al. These cone and hex splines of fractional and complex order generalize the univariate fractional and complex B-splines defined in \cite{article:ub,article:fbu} and investigated in, e.g., \cite{article:fm,article:mf}. Explicit time domain representations are derived for these splines on $3$-directional meshes.  We present some properties of these two multivariate spline families such as recurrence, decay and refinement.  Finally it is shown that a bivariate hex spline and its integer lattice translates form a Riesz basis of its linear span.
\end{abstract}
\section{Introduction}
We introduce an extension of cone and box splines to fractional and complex orders. Given a direction matrix or knot set, both types of multivariate splines are defined in the Fourier domain as ridges along the directions. These new families include as special cases the three-directional box splines introduced by Condat and Van De Ville \cite{article:condat} as well as the hex splines defined by Van De Ville et al. \cite{article:vandeville}. Whereas the latter multivariate splines provide integer smoothness along a direction, the new families allow for fractional smoothness along different directions thus defining \textit{continuous} families (with respect to smoothness) of functions. In the case of complex orders, cone and hex splines have an additional phase factor automatically built in that allows -- just as in the one-dimensional setting -- the enhancement of frequencies along the directions defined by the knot set. The importance of phase has been well known since the experiment by Oppenheimer and Lim \cite{article:ol}. For this reason, hex splines of complex orders may be potentially employed for multivariate signal or image analysis.

This paper, however, focuses on the mathematical foundations of cone and hex splines of fractional and complex orders. In the current setting, these two spline families generalize the univariate concepts of fractional and complex B-splines as defined by Blu and Unser \cite{article:ub}, respectively, Forster et al. \cite{article:fbu} to higher dimensions. Previously, an extension of univariate complex B-splines to higher dimensions was proposed in \cite{article:mf}. This approach is based on a generalized version of the Hermite--Genocchi formula and produces the analog of multivariate simplex splines (of complex orders). Our approach differs also from that undertaken in \cite{article:fm2} where the multivariate extension was done via ridge functions. The focus of our setting lies more on the geometry of the knot set.

The structure of this paper is as follows. In Section \ref{section:review}, we briefly review the definitions of cone and box splines and present those properties that will be relevant in subsequent sections. The next section deals with cone splines on $3$-directional meshes. There, we extend the results obtained in \cite{article:condat} and obtain an explicit time domain representation of such cone splines on $3$-directional meshes with different weights along the various directions. Hex splines with different weights along the given directions are defined in Section \ref{section:conesplines} and it is shown that the hex splines in \cite{article:condat,article:vandeville} are again special cases. Fractional cone and hex splines are introduced in Section \ref{section:fractionalconeshexes}.  While they are defined in the frequency domain, explicit formulas for them in the time domain are presented.  We show that fractional cone splines in $\mathbb{R}^s$ constructed from $s$ knots obey a recurrence formula that is similar to that derived by Micchelli \cite{article:micchelli} for polynomial cone splines.  In the final section, we derive decay rates for hex splines and also show that hex splines are refinable functions.  We also show that with modest assumptions placed on the knots, the set $V$ consisting of a bivariate hex spline and its integer lattice translates form a Riesz basis for $\spn(V)$.  We conclude the paper by extending the definitions of fractional splines to introduce cone and hex splines of complex order.

\section{Review of Cone and Box Splines} \label{section:review}
The $B$-spline of Curry and Schoenberg \cite{article:curry} is defined as follows:
\begin{equation}
    N_{0,n}(x) := n\left[ x_0,\ldots, x_n\right](\cdot - x)_+^{n-1}      \label{Bspline}
\end{equation}
where $x_0 \leq \cdots \leq x_n$ with $x_0 < x_n$ are the knots, $x_+ = \max(0,x)$ is the truncated power function, and $\left[ x_0,\ldots, x_n\right]$ is the divided difference operator.  Multiplication by $n$ normalizes the spline so that $\int\limits_{\mathbb{R}} N_{0,n}(x)\diff{x} = 1$.
\subsection{Cone Splines}
It is natural to define multivariate splines by first generalizing the truncated power function that appears in \eqref{Bspline}.  We have the following definition:
\begin{conedefn}
\label{defn:cone}
Let $\bX = \left\{ \bx^1, \ldots, \bx^s\right\} \subset \mathbb{R}^s$ be a linearly independent set of vectors with $\cone(\bX) = \left\{ t_1\bx^1+\cdots +t_s\bx^s \vb t_1, \ldots, t_s \geq 0\right\}$ and $M = \left[ \bx^1 \cdots \bx^s\right]$.  Then the \emph{cone spline} $\Cspline$ is defined as
\begin{equation}
    \Cspline := \left\{ \begin{array}{ll} \lvert \det(M)\rvert^{-1}, & \bx \in \cone(\bX); \\ 0, & \mbox{otherwise},\end{array}\right.     \label{conedefstart}
\end{equation}
and inductively for additional nonzero knots by
\begin{equation*}
    C\left(\bx\vb \bX \cup \left\{ \bx^{s+1}\right\}\right) = \int_0^{\infty} C(\bx - t\bx^{s+1} \vb \bX) \diff{t}.
\end{equation*}
\end{conedefn}
Definition \ref{defn:cone} leads to several properties (see \cite{book:chui}) of cone splines.  For $\bX = \left\{ \bx^1,\ldots, \bx^n\right\} \subset \mathbb{R}^s\backslash\left\{ \mathbf{0}\right\}$, $n\geq s$, $\spn(\bX) = \mathbb{R}^s$.  The first four properties appear in \cite{book:chui} and the last can be found in \cite{article:dahmen}:
\begin{itemize}
    \item $\supp(\Cspline) = \cone(\bX) := \left\{ t_1\bx^1 + \cdots + t_n\bx^n \vb t_1,\ldots, t_n \geq 0\right\}$.
    \item $\Cspline > 0$ for $\bx \in \cone(\bX)$.
    \item $\Cspline$ is a piecewise polynomial function of total degree $n-s$.
    \item Let $p$ be the minimum number of knots that can be removed from $\bX$ so that the resulting set does not span $\mathbb{R}^s$.  Then
        \begin{equation}
            \Cspline \in C^{p-2}\left(\mathbb{R}^s\right).          \label{conecontinuity}
        \end{equation}
    \item For all continuous functions $g: \mathbb{R}^s \rightarrow \mathbb{R}$ with compact support,
        \begin{equation}
            \int_{\mathbb{R}^s} \Cspline g(\bx) \diff{\bx} = \int_{[0,\infty)^n} g\left(\sum_{k=1}^n t_k\bx^k\right) \diff{\mathbf{t}}.  \label{conedistribution}
        \end{equation}
\end{itemize}
As the set $C_0(\mathbb{R}^s)$ of continuous functions with compact support contains the set $C_0^\infty$ of infinitely differentiable functions as a dense (in $L^p$, $p\in [1,\infty)$) subset and the latter is dense in the Schwartz space $\mathcal{S}(\mathbb{R}^s)$, we can interpret the cone spline $C(\bullet \vb \bX)$ as a tempered distribution by setting
\begin{equation*}\label{eq5}
\inn{C(\bullet \vb \bX)}{g} := \int_{[0,\infty)^n} g\left(\sum_{k=1}^n t_k\bx^k\right) \diff{\mathbf{t}}, \qquad\forall\,g\in \cS(\R^s).
\end{equation*}

We can use \eqref{eq5} to see that the Fourier transformation of the cone spline, considered as an element of $\cS'(\R^s)$, is given by
\begin{equation}
    \widehat{C}(\bomega \vb \bX) = \prod_{k=1}^n \left(\frac{1}{i\langle \bomega, \bx^k\rangle} + \pi \delta\left( \langle \bomega, \bx^k\rangle\right)\right),     \label{coneft}
\end{equation}
where $\prod$ represents the direct or tensor product of distributions. Indeed, for any $g\in \cS(\R^s)$,
\begin{align*}
\inn{\widehat{C}(\bomega \vb \bX)}{g} & = \inn{\Cspline}{\widehat{g}} =  \int_{[0,\infty)^n} \widehat{g}\left(\sum_{k=1}^n t_k\bx^k\right) \diff{\mathbf{t}} = \inn{\bH}{\widehat{g}(\bX \bt)} = \inn{\widehat{\bH}(\bX^T\bomega)}{g}\\
& = \inn{\prod_{k=1}^n \widehat{H}(\inn{\bx^k}{\bomega})}{g} = \inn{\prod_{k=1}^n \left(\frac{1}{i\langle \bomega, \bx^k\rangle} + \pi \delta\left( \langle \bomega, \bx^k\rangle\right)\right)}{g},
\end{align*}
where we set $\bX \bt := \sum\limits_{k=1}^n t_k\bx^k$, and where $H$ and $\bH$ denotes the one-, respectively, $n$-dimensional Heaviside function.
\begin{generalizedremark}
In the sequel, we need to consider direct products of generalized functions $S$ and $T$ of the form $S^m$ and $S^m T^n$, $m,n\in \N$. These are defined as follows \cite{book:gelfand}.
\begin{equation*}
    \inn{S^m}{\phi(\bx)} := \inn{S_{x_1}}{\inn{S_{x_2}}{\cdots\langle S_{x_m}, \phi(x_1,\ldots,x_m)\rangle \cdots}},
\end{equation*}
where the subscript on $S$ indicates on which variable $S$ acts. Similarly, one defines
\begin{equation*}
    \inn{S^m T^n}{\phi (\bx,\by)} := \inn{S^m}{\inn{T^n}{\phi (\bx,\by}},
\end{equation*}
where $S^m$, respectively, $T^n$ acts on the $m$ variables $\bx = (x_1, \ldots, x_m)$, respectively, the $n$ variables $\by = (y_1, \ldots, y_n)$.
\end{generalizedremark}
The distributional relation \eqref{conedistribution} can also be used to show that if $A$ is any nonsingular $s\times s$ matrix, then
\begin{equation}
    C\left(\bx \vb \left\{ A\bx^1, \ldots, A\bx^n\right\}\right) = \frac{1}{\lvert\det(A)\rvert} C\left( A^{-1}\bx \vb \bX\right).       \label{conerotateknots}
\end{equation}
Finally, Micchelli \cite{article:micchelli} proved the following recurrence formula obeyed by cone splines:
\begin{equation}
    \Cspline = \frac{1}{n-s} \sum_{k=1}^n \lambda_k C\left(\bx \vb \bX\backslash\left\{ \bx^k\right\}\right)     \label{conerecurrence}
\end{equation}
where $\bx = \sum\limits_{k=1}^n \lambda_k \bx^k$.
Two cone splines are plotted in Figure \ref{fig:conesplines}.
\begin{figure}[ht]
    \begin{center}
        \resizebox{4in}{!}{\includegraphics{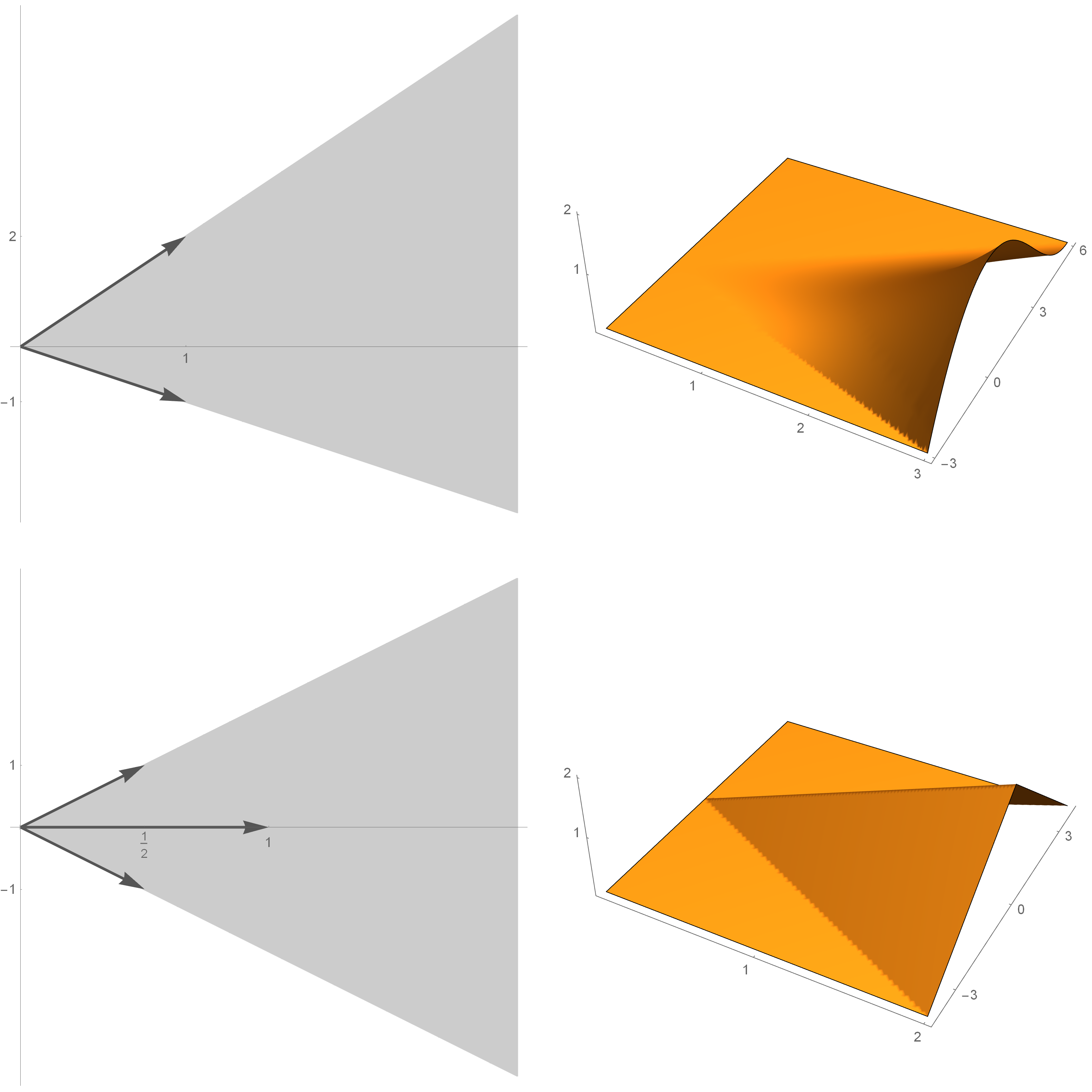}}
    \end{center}
    \caption{Two cone splines and their support.  Let $\bx^1 = (1,2)^T$ and $\bx^2=(1,-1)^T$.  Then the knot set is $\bX = \left\{ \bx^1,\bx^1,\bx^2,\bx^2,\bx^3\right\}$ and the top spline is a piecewise cubic polynomial.  The spline on the bottom is piecewise linear with knots $\bx^1 = \left(\frac{1}{2},1\right)^T$,  $\bx^2 = \left(\frac{1}{2},-1\right)^T$, and $\bx^3 = (1,0)^T$.}
    \label{fig:conesplines}
\end{figure}
\subsection{Box Splines}
Multivariate box splines are defined in a manner similar to that for cone splines:
\begin{boxdefn}
\label{defn:box}
Let $\bX = \left\{ \bx^1, \ldots, \bx^s\right\} \subset \mathbb{R}^s$ be a linearly independent set of vectors. Let $M = \left[ \bx^1 \cdots \bx^s\right]$ and denote by $Z(\bX)$ the zonotope $\left\{ t_1\bx^1+\cdots +t_s\bx^s \vb t_1, \ldots, t_s \in [0,1]\right\}$.
Then the \emph{box spline} $\Bspline$ is defined as
\begin{equation*}
    \Bspline := \left\{ \begin{array}{ll} \lvert \det(M)\rvert^{-1}, & \bx \in \cube(\bX); \\ 0, & \mbox{otherwise},\end{array}\right.
\end{equation*}
and inductively for additional nonzero knots by
\begin{equation*}
    B\left(\bx\vb \bX \cup \left\{ \bx^{s+1}\right\}\right) = \int_0^1 B(\bx - t\bx^{s+1} \vb \bX) \diff{t}.
\end{equation*}
\end{boxdefn}
Definition \eqref{defn:box} gives rise to several properties (see \cite{book:chui}) of box splines.  For $\bX = \left\{ \bx^1,\ldots, \bx^n\right\} \subset \mathbb{R}^s\backslash\left\{ \mathbf{0}\right\}$, $n\geq s$, $\spn(\bX) = \mathbb{R}^s$ (see \cite{book:chui}):
\begin{itemize}
    \item $\supp(\Bspline) = \cube(\bX) = \left\{ t_1\bx^1 + \cdots + t_n\bx^n \vb t_1,\ldots, t_n \in [0,1] \right\}$.
    \item $\Bspline > 0$ for $\bx \in \cube(\bX)$.
    \item $\Bspline$ is a piecewise polynomial function of total degree $n-s$.
    \item For all continuous functions $g: \mathbb{R}^s \rightarrow \mathbb{R}$ with compact support,
        \begin{equation}
            \int_{\mathbb{R}^s} \Bspline g(\bx) \diff{\bx} = \int_{[0,1]^n} g\left(\sum_{k=1}^n t_k\bx^k\right) \diff{\mathbf{t}}.  \label{boxdistribution}
        \end{equation}
\end{itemize}
As $B(\bullet \vb \bX)\in L^1(\R^s)$, we interpret $B(\bullet \vb \bX)$ as a regular tempered distribution acting on Schwartz functions via \eqref{boxdistribution}. Employing similar arguments to those given in the derivation of the Fourier transform of a cone spline, replacing the $n$-dimensional Heaviside function by the characteristic function of the unit $n$-cube, one shows that $\widehat{B}(\bullet \vb \bX)$ is given by
\begin{equation}
    \widehat{B}(\bomega \vb \bX) = \prod_{k=1}^n \frac{\left(1-e^{-i\langle \bomega, \bx^k\rangle}\right)}{i\langle \omega, \bx^k\rangle}.      \label{boxft}
\end{equation}
Writing the above equation in the form
\begin{align}\label{10}
\widehat{B}(\bomega \vb \bX) &= \prod_{k=1}^n \left(1-e^{-i\langle \bomega, \bx^k\rangle}\right) \cdot \prod_{k=1}^n \frac{1}{i\langle \omega, \bx^k\rangle} = \prod_{k=1}^n \left(1-e^{-i\langle \bomega, \bx^k\rangle}\right) \cdot \prod_{k=1}^n \left(\frac{1}{i\langle \omega, \bx^k\rangle} + \pi \delta\left( \langle \bomega, \bx^k\rangle\right)\right)\nonumber\\
& =: \widehat{\nabla}_\bX \widehat{C}(\bomega \vb \bX),
\end{align}
where we set
\[
\widehat{\nabla}_\bX := \prod_{k=1}^n \left(1-e^{-i\langle \bomega, \bx^k\rangle}\right)
\]
and used the fact that $\left(1-e^{-i\langle \bomega, \bx^k\rangle}\right)\cdot\delta\left( \langle \bomega, \bx^k\rangle\right) = 0$ (in $\cS'(\R^s)$). Taking the inverse Fourier transform of \eqref{10} and noting that $\widehat{\nabla}_\bX$ is the Fourier transform of the multivariate backwards difference operator $\frac{1}{\Gamma (n)}\,{\nabla_\bX}$ defined by
\[
\nabla_\bX := \prod_{\bx\in \bX} \nabla_\bx,
\]
where the backwards difference operator in the direction $\bx$ for a function $f:\R^s\to \R$ is given as
\[
\nabla_\bx f := f - f(\bullet - \bx),
\]
we obtain the well known identity
\beq\label{BC}
\Bspline = \frac{1}{\Gamma(n)}\, \nabla_\bX {C}(\bx \vb \bX).
\eeq
%
%
%
%
We end this section with examples of box splines plotted in Figure \ref{fig:boxsplines}.
\begin{figure}[ht]
    \begin{center}
        \resizebox{4in}{!}{\includegraphics{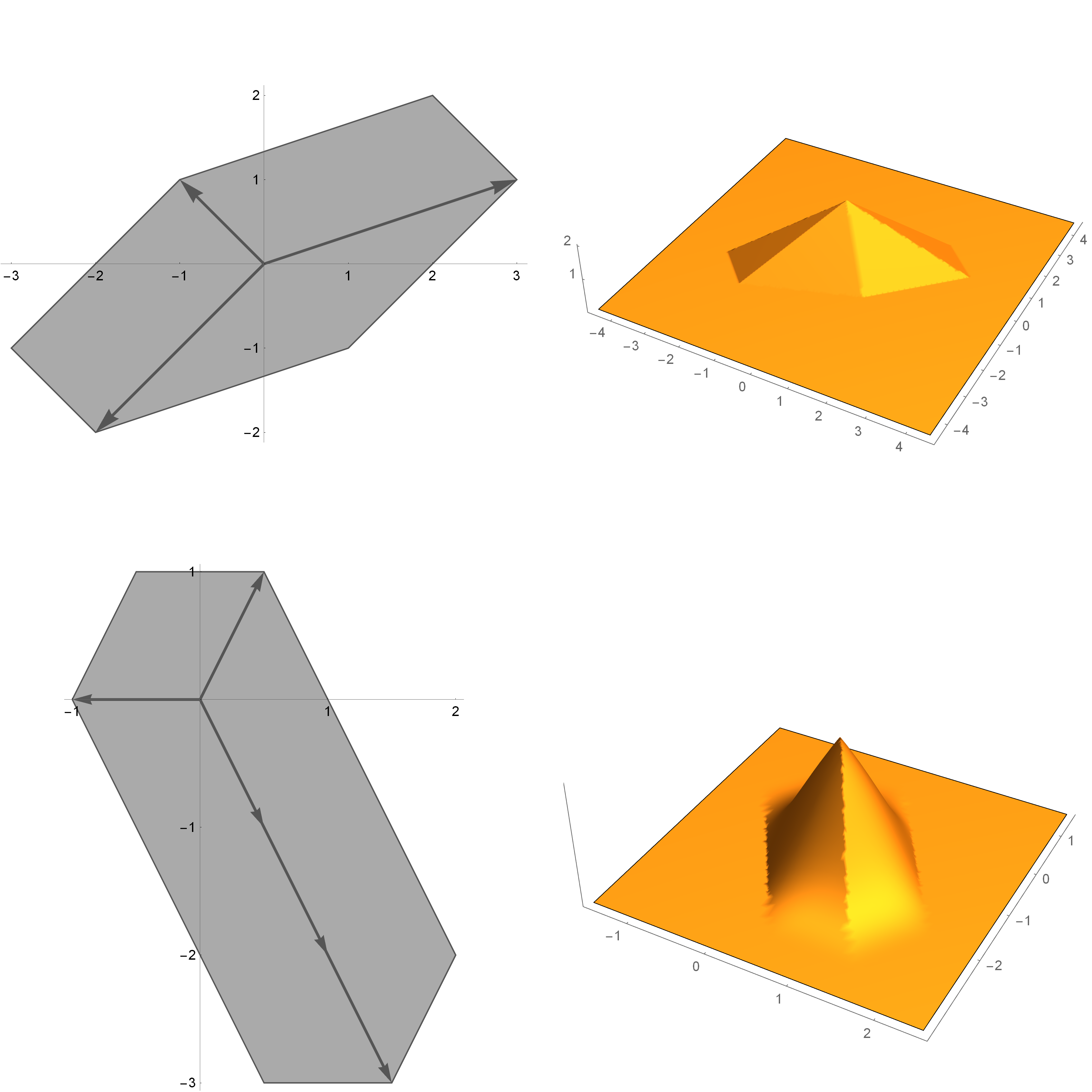}}
    \end{center}
    \caption{Two box splines and their support.  The top spline is a piecewise linear polynomial.  The knots are $\bx^1 = (3,1)^T$ and $\bx^2=(-2,-2)^T$ and $\bx^3=(-1,1)^T$.  The spline on the bottom is piecewise cubic.  If $\bx^1=\left(\frac{1}{2},1\right)^T$, $\bx^2=\left(\frac{1}{2},-1\right)^T$ and $\bx^3 = (-1,0)^T$, the knot set is $\bX = \left\{ \bx^1, \bx^2, \bx^2, \bx^2, \bx^3\right\}$.}
    \label{fig:boxsplines}
\end{figure}
\section{Cone Splines on $3$-Directional Meshes}
\label{section:conesplines}
A knot set $\bX = \left\{ \bx^1, \bx^2, \bx^3\right\} \subset \mathbb{R}^2$ will be called a \textit{$3$-directional mesh} if $\bx^1, \bx^2$ are linearly independent and $\bx^3 = \bx^1 + \bx^2$.

We are interested in developing explicit formulas for cone splines on $3$-directional meshes.  In particular, we seek formulas for cone splines of the form
\begin{equation*}
    \Cnspline := C_{\left(n_1,n_2,n_3\right)}\left(\bx \vb \bX\right):= C\left(\bx\; \Bigg\lvert \left\{ \underbrace{\strut \bx^1, \ldots \bx^1}_{n_1 \ \mbox{\scriptsize{times}}}, \underbrace{\bx^2, \ldots, \bx^2}_{n_2 \ \mbox{\scriptsize{times}}}, \underbrace{\bx^3, \ldots \bx^3}_{n_3 \ \mbox{\scriptsize{times}}} \right\}\right)
\end{equation*}
where $n_k \in \Z_+ := \{n\in \Z \vb n > 0\}$, $k=1,2,3$.

For the special case $\mathbf{n} = \left( n, n, n\right)$, $n \in \Z_+$ and knot set $\bX = \left\{ \bx^1, \bx^2, \bx^3\right\}$ where
\begin{equation*}
    \bx^1 = \left(\begin{array}{c} \frac{1}{2} \\ -\frac{\sqrt{3}}{2} \end{array}\right), \qquad \bx^2 = \left(\begin{array}{c} \frac{1}{2} \\ \frac{\sqrt{3}}{2} \end{array}\right), \qquad  \bx^3 = \bx^1 + \bx^2 = \left(\begin{array}{c} 1 \\ 0 \end{array}\right),
\end{equation*}
Condat and Van De Ville \cite{article:condat} were able to derive a closed formula for $\Cnspline$:
\begin{align}
    \Cnspline &= \frac{2}{\sqrt{3}}\sum_{k=0}^{n-1} \binom{n-1+k}{k} \frac{1}{(n-1-k)!(2n-1+k)!}  \nonumber \\
        &{}  \qquad\qquad \times \left(\frac{2\lvert x_2\rvert}{\sqrt{3}}\right)^{n-1-k} \left(x_1 - \frac{\lvert x_2\rvert}{\sqrt{3}}\right)_+^{2m-1+k}.         \label{condatcone}
\end{align}

It is straightforward to generalize \eqref{condatcone} for the case when the knots of $\bX$ are not repeated the same amount of times.  In order to derive such a formula for $\mathbf{n} = \left( n_1, n_2, n_3\right)$, we first need the following result for $s$-dimensional cone splines built from $s$ knots.
\begin{conesdirection}
                \label{prop:conesdirection}
Let $\bX = \left\{ \bx^1, \ldots, \bx^s \right\} \subset \mathbb{R}^s$ be a linearly independent set of vectors and suppose $\mathbf{n} = \left( n_1, \ldots, n_s\right) \in \mathbb{Z}^s_+$.  Then
\begin{equation}
    \Cnspline = \lvert \det(M)\rvert^{-1} \prod_{k=1}^s \frac{\langle \bx, \bxd^k\rangle_+^{n_k-1}}{\left(n_k-1\right)!}        \label{conespoints}
\end{equation}
where $M = \left[ \bx^1 \cdots \bx^s \right]$ and $\left\{ \bxd^1, \ldots, \bxd^s\right\}$ is the dual basis to $\bX$.
\end{conesdirection}
\begin{proof}
Let $\bE = \left\{ \be^1,\ldots, \be^s\right\}$ be the canonical basis for $\mathbb{R}^s$.  Using \eqref{coneft}, we have
\begin{equation*}
    \CnsplineE = \prod_{k=1}^s \left( \frac{1}{i\langle \bomega, \be^k\rangle} + \pi\delta\left(\langle \bomega, \be^k\rangle\right)\right)^{n_k}
                    = \prod_{k=1}^s \left( \frac{1}{i\omega_k} + \pi\delta\left(\omega_k\right)\right)^{n_k}.
\end{equation*}
%
%
Applying the inverse Fourier transform to the above equation, noting that each factor in the $n_k$-fold direct product produces the truncated power function $\left(x_k\right)_+^0$, we obtain the $n_k$-fold convolution product $\left(x_k\right)_+^0 * \cdots * \left(x_k\right)_+^0$. It is easy to establish that this product is $\frac{1}{\left(n_k-1\right)!}\left(x_k\right)_+^{n_k-1}$.  Thus we have
\begin{equation*}
    \CnsplineE = \prod_{k=1}^s \frac{\left(x_k\right)_+^{n_k-1}}{\left(n_k-1\right)!}.
\end{equation*}
Using \eqref{conerotateknots}, with $A = M$, and the fact that $\Cnspline = C_{\mathbf{n}}\left(\bx \vb \left\{ M\be^1,\ldots, M\be^s\right\}\right)$, we can write
\begin{equation*}
    \Cnspline = \frac{1}{\lvert\det(M)\rvert} C_{\mathbf{n}}\left(M^{-1}\bx \vb \bE\right) = \frac{1}{\lvert\det(M)\rvert} \prod_{k=1}^s \frac{\langle \bx, \bxd^k\rangle}{\left(n_k-1\right)!}
\end{equation*}
and the proof is complete.
\end{proof}
We can use Proposition \ref{prop:conesdirection} with $s=2$ to find a closed formula for cone splines whose knot set is a $3$-directional mesh.  We need the following lemma:
\begin{iteratedintegral}
    \label{lemma:iteratedintegral}
For $n,m,\ell \in \Z_+$ and $x,a,b \in \mathbb{R}$ with $b < a < x$, let
\begin{equation*}
    F_{(n,m,\ell)}(x,a,b) = \frac{\scz{1}}{\scz{(n-1)!(m-1)!(\ell-1)!}}\int_a^x (u-b)^{n-1}(u-a)^{m-1}(x-u)^{\ell-1}\diff{u}.
\end{equation*}
Then
\begin{equation*}
    F_{(n,m,\ell)}(x,a,b) = \sum_{k=0}^{n-1} \binom{\scz{m+k-1}}{\scz{k}} \frac{\scz{1}}{\scz{(n-1-k)!(m+\ell+k-1)!}}(a-b)^{n-1-k}(x-a)_+^{m+\ell+k-1}.
\end{equation*}
\end{iteratedintegral}
\begin{proof}
The proof is a straightforward application of the result
\begin{equation}
    \frac{1}{\Gamma(\alpha)} \int_a^x (u-a)^{\beta -1}(u-b)^{\gamma-1} (x-u)^{\alpha-1}\diff{u} = (a-b)^{\gamma-1}\frac{\Gamma(\beta)}{\Gamma(\alpha + \beta)}(x-a)^{\alpha+\beta-1}{}_2F_1\left(1-\gamma, \beta; \alpha+\beta; - \frac{x-a}{a-b}\right)       \label{stamko}
\end{equation}
that appears in \cite[p.~41]{book:samko} where $b < a < x$, $\alpha > 0$, $\mbox{Re}(\beta) > 0$, $\gamma \in \mathbb{C}$.  In our case, we insert positive integers $\gamma = n$, $\beta = m$, and $\alpha = \ell$ to write
\begin{align}
    F_{(n,m,\ell)}(x,a,b) &= \frac{1}{(n-1)!(m+\ell-1)!}(a-b)^{n-1}(x-a)^{m+\ell-1}{}_2F_1\left(1-n,m;m+\ell; - \frac{x-a}{a-b}\right) \nonumber \\
        &= \frac{1}{(n-1)!(m+\ell-1)!}\sum_{k=0}^\infty \frac{(1-n)_k (m)_k}{(m+\ell)_k} \frac{(-1)^k}{k!} (x-a)^{m+\ell+k-1}(a-b)^{n-1-k}. \label{2f1}
\end{align}
The rising Pochhammer symbols in the ${}_2F_1$ hypergeometric function are given by
\begin{align*}
    (1-n)_k &= (1-n)(2-n)\cdots (k-n) = (-1)^k \frac{(n-1)!}{(n-1-k)!} \\
    (m)_k &= m(m+1)\cdots (m+k-1) = \frac{(m+k-1)!}{(m-1)!} \\
    (m+\ell)_k &= \frac{(m+\ell+k-1)!}{m+\ell-1)!}.
\end{align*}
Inserting these identities into \eqref{2f1}, simplifying and noting $(x-a) = (x-a)_+$ gives the desired result.
%
%
%
%
%
%
%
%
%
\end{proof}
We are now ready to state a closed formula for cone splines on $3$-directional meshes:
\begin{cone3direction}
    \label{prop:cone3direction}
Suppose $\bX \subset \mathbb{R}^2$ is the $3$-directional knot set $\bX = \left\{ \bx^1, \bx^2, \bx^3\right\}$ with $\bx^1, \bx^2$ linearly independent and $\bx^3 = \bx^1 + \bx^2 = (1,0)^T$ and, without loss of generality, $x^1_2 > 0$.  Then the cone spline $\Cnspline$ is given by
\renewcommand{\arraystretch}{2.5}
\begin{align*}
    \Cnspline 
    &= \frac{1}{x^1_2} \left\{\begin{array}{ll} \sum\limits_{k=0}^{n_1-1} \binom{n_2+k-1}{k} \frac{1}{\left(n_1-1-k\right)!\left(n_2+n_3+k-1\right)!}\left(\frac{x_2}{x^1_2}\right)^{n_1-1-k}\left(x_1 - \frac{x^1_1}{x^1_2}x_2\right)_+^{n_2+n_3+k-1}, & x_2 \geq 0 \\
     \sum\limits_{k=0}^{n_2-1} \binom{n_1+k-1}{k} \frac{1}{\left(n_2-1-k\right)!\left(n_1+n_3+k-1\right)!}\left(-\frac{x_2}{x^1_2}\right)^{n_2-1-k}\left(x_1 - \frac{x^1_1-1}{x^1_2}x_2\right)_+^{n_1+n_3+k-1}, & x_2 < 0
    \end{array}\right..
\end{align*}
\renewcommand{\arraystretch}{1}
%
%
%
\end{cone3direction}
\begin{proof}
From \eqref{coneft}, we know that
\begin{equation*}
    \widehat{C}_{\mathbf{n}}(\bomega \vb \bX) = \widehat{C}_{\left(n_1,n_2\right)}\left(\bomega \vb \left\{ \bx^1, \bx^2\right\}\right) \cdot \left( \frac{1}{i\omega_1} + \pi \delta\left(\omega_1\right)\right)^{n_3}
\end{equation*}
so that
\begin{align}
    \Cnspline &= C_{\left(n_1,n_2\right)}\left(\bx \vb \left\{\bx^1, \bx^2\right\}\right) * \frac{\left(x_1\right)_+^{n_3-1}}{\left(n_3-1\right)!}\delta\left(x_2\right) \nonumber \\
        &= \left( \vert\det(M)\rvert \prod_{k=1}^3 \left(n_k-1\right)!\right)^{-1} \nonumber \\
        &  \qquad \qquad \times \int_{-\infty}^{x_1} \langle \left(u,x_2\right), \bxd^1\rangle_+^{n_1-1}\langle \left(u,x_2\right), \bxd^2\rangle_+^{n_2-1}\left(x_1-u\right)^{n_3-1}\diff{u}
                    \label{coneconvolve}
\end{align}
Since $\bx^1 + \bx^2 = (1,0)^T$, we know $\bx^2 = \left( 1-x^1_1, -x^1_2\right)^T$ and from here we can easily show that the dual basis is
\begin{equation}
    \left\{ \bxd^1, \bxd^2\right\} = \left\{ \left( \begin{array}{c} 1 \\ \frac{1-x^1_1}{x^1_2}\end{array}\right), \left(\begin{array}{c} 1 \\ -\frac{x^1_1}{x^1_2}\end{array}\right)\right\}  \label{dualvectors}
\end{equation}
and $\lvert \det(M)\rvert = x^1_2$.  Inserting these identities into \eqref{coneconvolve} and setting $K = \left( x^1_2\prod_{k=1}^3 \left(n_k-1\right)!\right)^{-1}$ gives
\begin{equation}
   \Cnspline = K\int_{-\infty}^{x_1} \left( u - \left(\frac{x_1^1-1}{x^1_2}\right) x_2\right)_+^{n_1-1}\left( u -\frac{x_1^1}{x_2^1} x_2\right)_+^{n_2-1}\left(x_1-u\right)^{n_3-1}\diff{u}.
                    \label{coneasintegral}
\end{equation}
If $x_2 > 0$, we have $\frac{x_1^1-1}{x^1_2}x_2 < \frac{x^1_1}{x^1_2}x_2 < x_1$ and \eqref{coneasintegral} becomes
\begin{equation*}
   \Cnspline = K\int\limits_{\frac{x_1^1}{x_2^1} x_2}^{x_1} \left( u - \left(\frac{x_1^1-1}{x^1_2}\right)x_2\right)^{n_1-1}\left( u - \frac{x_1^1}{x_2^1} x_2\right)^{n_2-1}\left(x_1-u\right)^{n_3-1}\diff{u}
\end{equation*}
and we can use Lemma \ref{lemma:iteratedintegral} to complete the proof.  The case where $x_2 < 0$ is similar and we can use the continuity of the cone spline \eqref{conecontinuity} for the case $x_2=0$.
        %
%
%

\end{proof}
\begin{transformationremark}
Using \eqref{conerotateknots}, we can construct a formula for the cone spline with the general $3$-directional mesh $\bx^1, \bx^2, \bx^1+\bx^2$.
\end{transformationremark}
\begin{condatremark}
Setting $\mathbf{n} = (n,n,n)$ in Proposition \ref{prop:cone3direction} gives the cone spline \eqref{condatcone} of Condat and Van De Ville.
\end{condatremark}

We use Proposition \ref{prop:cone3direction} to plot the cone splines shown in Figure \ref{fig:cone3splines}.
\begin{figure}[ht]
    \begin{center}
        \resizebox{4in}{!}{\includegraphics{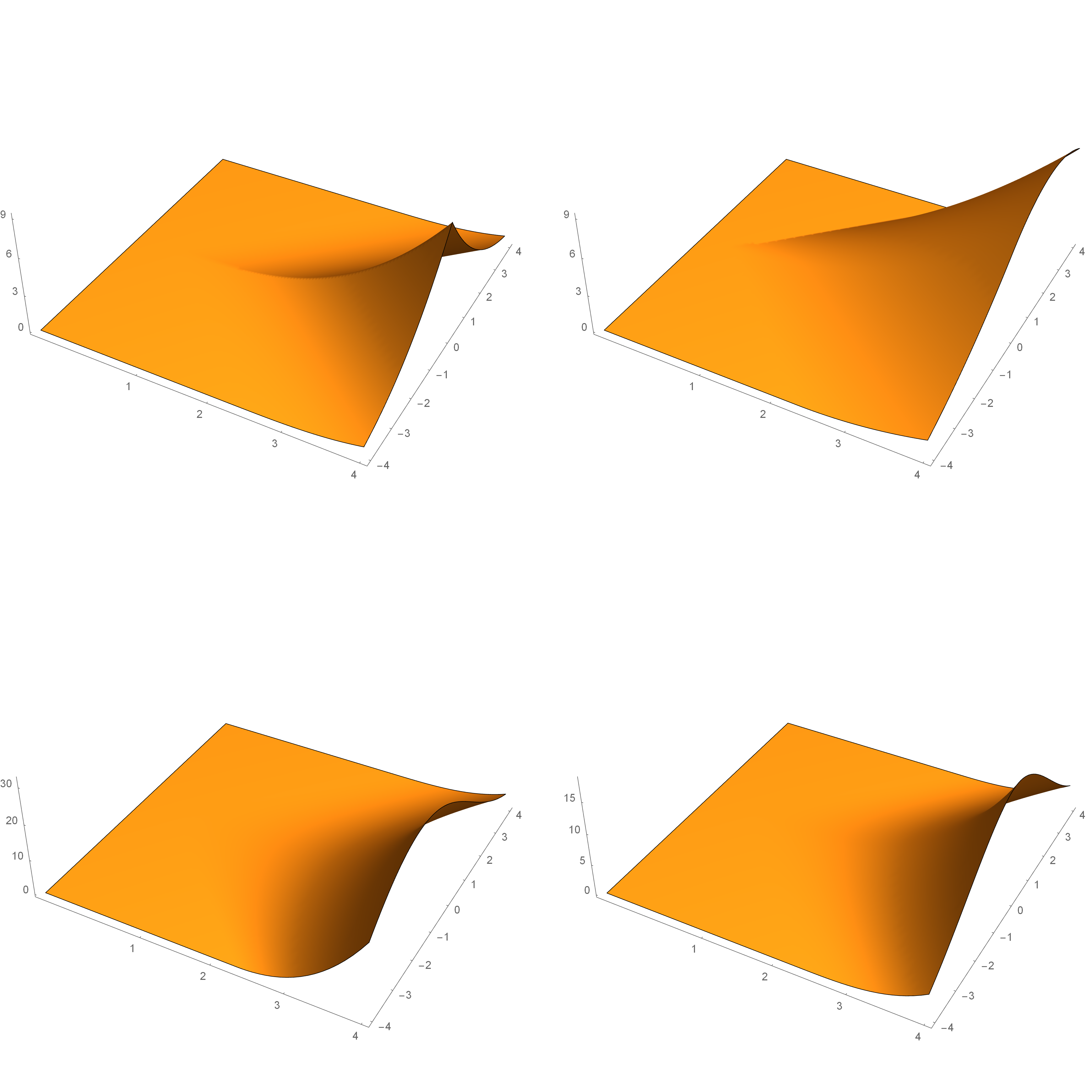}}
    \end{center}
    \caption{For $\bX = \left\{ \bx^1,\bx^2,\bx^3\right\}$ where $\bx^1 = \left(\frac{1}{2},\frac{\sqrt{3}}{2}\right)^T$, $\bx^2 = \left(\frac{1}{2},-\frac{\sqrt{3}}{2}\right)^T$ and $\bx^3 = (1,0)^T$, the cone splines plotted clockwise from the top left are $C_{(1,1,3)}(\bx \vb \bX)$, $C_{(2,1,1)}(\bx \vb \bX)$, $C_{(2,3,1)}(\bx \vb \bX)$, and $C_{(2,2,2)}(\bx \vb \bX)$.}
    \label{fig:cone3splines}
\end{figure}
\section{Hex Splines}
\label{section:hexsplines}
In this section we give a general formula for the so-called \textit{hex splines} introduced by Van De Ville et. al. \cite{article:vandeville}.  These splines were developed for applications that involve hexagonally sampled data.  The hex splines can be defined by using box splines with the knot set $\bXm = \left\{ \bx^1, \bx^2, -\bx^3\right\}\subset \mathbb{R}^2$ where $\bx^3 = \bx^1 + \bx^2$.  For this knot set and $\mathbf{n} = \left(n_1, n_2, n_3\right) \in \mathbb{Z}^3_+$, our hex spline takes the form
\begin{equation*}
    \Bnsplinem := B\left(\bx \;\Bigg\lvert \left\{ \underbrace{\strut \bx^1, \ldots \bx^1}_{n_1 \ \mbox{\scriptsize{times}}}, \underbrace{\bx^2, \ldots, \bx^2}_{n_2 \ \mbox{\scriptsize{times}}}, \underbrace{-\bx^3, \ldots -\bx^3}_{n_3 \ \mbox{\scriptsize{times}}} \right\}\right).
\end{equation*}
The resulting region of support of the spline is then a hexagon -- see Figure \ref{fig:hexsplinesupport}.
\begin{figure}[ht]
    \begin{center}
        \resizebox{\textwidth}{!}{\includegraphics{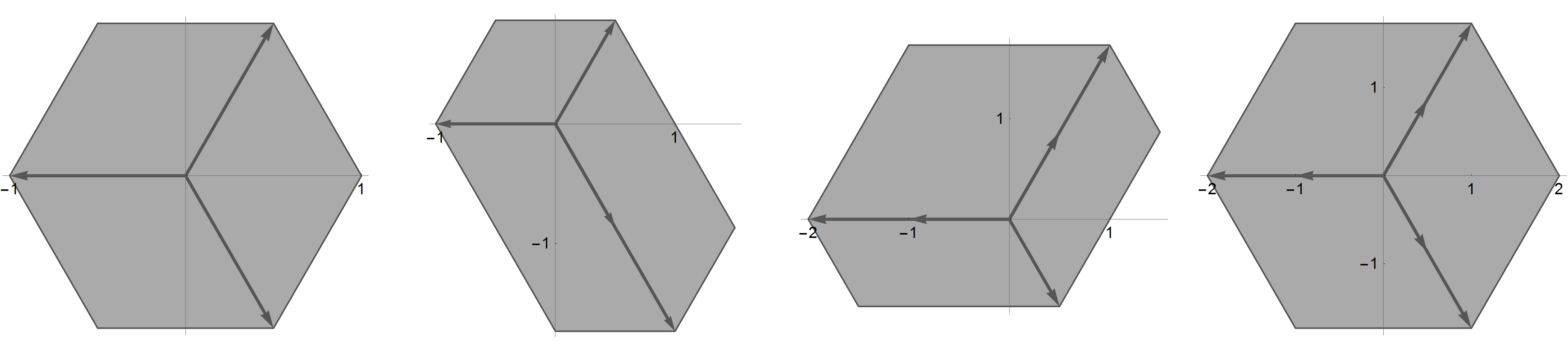}}
    \end{center}
    \caption{Support for hex splines (left to right) $B_{(1,1,1)}(\bx \vb \bXm)$, $B_{(2,1,1)}(\bx \vb \bXm)$, $B_{(1,2,2)}(\bx \vb \bXm)$, and $B_{(2,2,2)}(\bx \vb \bXm)$ using the knots $\bx^1 = \left(\frac{1}{2}, -\frac{\sqrt{3}}{2}\right)^T$, $\bx^2 = \left(\frac{1}{2}, \frac{\sqrt{3}}{2}\right)^T$, and $-\bx^3 = -\bx^1-\bx^2 = (-1,0)^T$.}
    \label{fig:hexsplinesupport}
\end{figure}
Consider the Fourier transform \eqref{boxft} of the box spline:
\begin{align}
    \widehat{B_{\mathbf{n}}}(\bomega \vb \bXm) &= \frac{\left( 1 - e^{-i\langle \bomega, \bx^1\rangle}\right)^{n_1}\left( 1 - e^{-i\langle \bomega, \bx^2\rangle}\right)^{n_2}\left( 1 - e^{i\langle \bomega, \bx^3\rangle}\right)^{n_3}}{\left(i\langle \bomega, \bx^1\rangle\right)^{n_1}\left(i\langle \bomega, \bx^2\rangle\right)^{n_2}\left(i\langle \bomega, -\bx^3\rangle\right)^{n_3}} \nonumber \\
     & \nonumber \\
     &=\frac{\left( 1 - e^{-i\langle \bomega, \bx^1\rangle}\right)^{n_1}\left( 1 - e^{-i\langle \bomega, \bx^2\rangle}\right)^{n_2}\left( e^{i\langle \bomega, \bx^3\rangle}-1\right)^{n_3}}{\left(i\langle \bomega, \bx^1\rangle\right)^{n_1}\left(i\langle \bomega, \bx^2\rangle\right)^{n_2}\left(i\langle \bomega, \bx^3\rangle\right)^{n_3}} \\
     & \nonumber \\
     &= \left( \frac{\left( 1 - e^{-i\langle \bomega, \bx^1\rangle}\right)}{i\langle \bomega, \bx^1\rangle}\right)^{n_1} \left( \frac{\left( 1 - e^{-i\langle \bomega, \bx^2\rangle}\right)}{i\langle \bomega, \bx^2\rangle}\right)^{n_2}\left( \frac{\left( e^{i\langle \bomega, \bx^3\rangle}-1\right)}{i\langle \bomega, \bx^3\rangle}\right)^{n_3}
            \label{boxadddeltas}
\end{align}
Using arguments from \cite[Section 3.11]{book:gelfand} and \cite[Theorem 7.1.16]{book:hormander}, we can rewrite each of the factors above as
\begin{equation*}
\left( \frac{\left( 1 - e^{-i\langle \bomega, \bx^k\rangle}\right)}{i\langle \bomega, \bx^k\rangle}\right)^{n_k} =
\left( 1 - e^{-i\langle \bomega, \bx^k\rangle}\right)^{n_k}\,\left( \frac{1}{i\langle \bomega, \bx^k\rangle} + \pi \delta\left(\langle \bomega,\bx^k\rangle\right)\right)^{n_k},
%
%
%
%
%
%
\end{equation*}
for $k=1,2$ (and a slight modification when $k=3$) so that \eqref{boxadddeltas} becomes
\begin{equation}
    \widehat{B_{\mathbf{n}}}(\bomega \vb \bXm) = \left( 1 - e^{-i\langle \bomega, \bx^1\rangle}\right)^{n_1}\left( 1 - e^{-i\langle \bomega, \bx^2\rangle}\right)^{n_2}  \left( e^{i\langle \bomega, \bx^3\rangle} - 1\right)^{n_3} \widehat{C_{\mathbf{n}}}(\bomega \vb \bX). \label{boxftviaconeft}
\end{equation}
Taking inverse Fourier transforms gives
\begin{equation}
    \Bnsplinem = f(\bx) * \Cnspline          \label{boxconvolve}
\end{equation}
where $f(\bx)$ is the inverse Fourier transform of
\begin{equation}
    \hat{f}(\bomega) =  \left( 1 - e^{-i\langle \bomega, \bx^1\rangle}\right)^{n_1}\left( 1 - e^{-i\langle \bomega, \bx^2\rangle}\right)^{n_2}  \left( e^{i\langle \bomega, \bx^3\rangle} - 1\right)^{n_3}.     \label{expproduct}
\end{equation}
As noted in \cite{article:condat}, if we can write \eqref{expproduct} as
\begin{equation*}
    \hat{f}(\omega) = \sum_{\mathbf{k}\in\mathbb{Z}^2} c_{k_1,k_2} e^{-i\langle \bomega, M\mathbf{k}\rangle}
\end{equation*}
where $M = \left[ \bx^1 \ \bx^2 \right]$, then
\begin{equation*}
    f(\bx) = \sum_{\mathbf{k}\in\mathbb{Z}^2} c_{k_1,k_2} \delta(\bx - M\mathbf{k})
\end{equation*}
so that \eqref{boxconvolve} becomes
\begin{equation}
    \Bnsplinem = \sum_{\mathbf{k}\in\mathbb{Z}^2} c_{k_1,k_2} C_{\mathbf{n}}(\bx - M\mathbf{k}).       \label{boxsum}
\end{equation}
In the case where $\mathbf{n} = (n, n, n)$, $n \in \Z_+$, and $\bx^3 = \bx^1 + \bx^2 = (1,0)^T$, Condat and Van De Ville \cite{article:condat} found the following formula for the coefficients in \eqref{boxsum}:
\begin{equation}
    c_{k_1,k_2} = \sum_{k=\max\left(k_1,k_2,0\right)}^{\min\left(n+k_1,n+k_2,n\right)} (-1)^{k_1+k_2+k} \binom{n}{k-k_1}\binom{n}{k-k_2}\binom{n}{k},       \label{condatcoeffs}
\end{equation}
and combining \eqref{condatcoeffs} and \eqref{boxsum} with the cone spline \eqref{condatcone}, they found the following formula for the hex spline:
\begin{align}
    \Bnspline &= \sum_{k_1,k_2 = -n}^n c_{k_1,k_2} \sum_{k=0}^{n-1} \binom{n-1+k}{k} \frac{1}{(2n-1+k)!(n-1-k)!} \nonumber \\
        & \quad  \times \vb \frac{2x_2}{\sqrt{3}} + k_1 - k_2\vb^{n-1-k} \left(x_1 - \frac{k_1+k_2}{2} - \vb \frac{x_2}{\sqrt{3}} + \frac{k_1-k_2}{2}\vb\right)_+^{2n-1+k}.      \label{condatbox}
\end{align}
Figure \ref{fig:condatsplines} shows the hex splines of Condat and Van De Ville for $n=1,2$.
\begin{figure}[ht]
    \begin{center}
        \resizebox{4in}{!}{\includegraphics{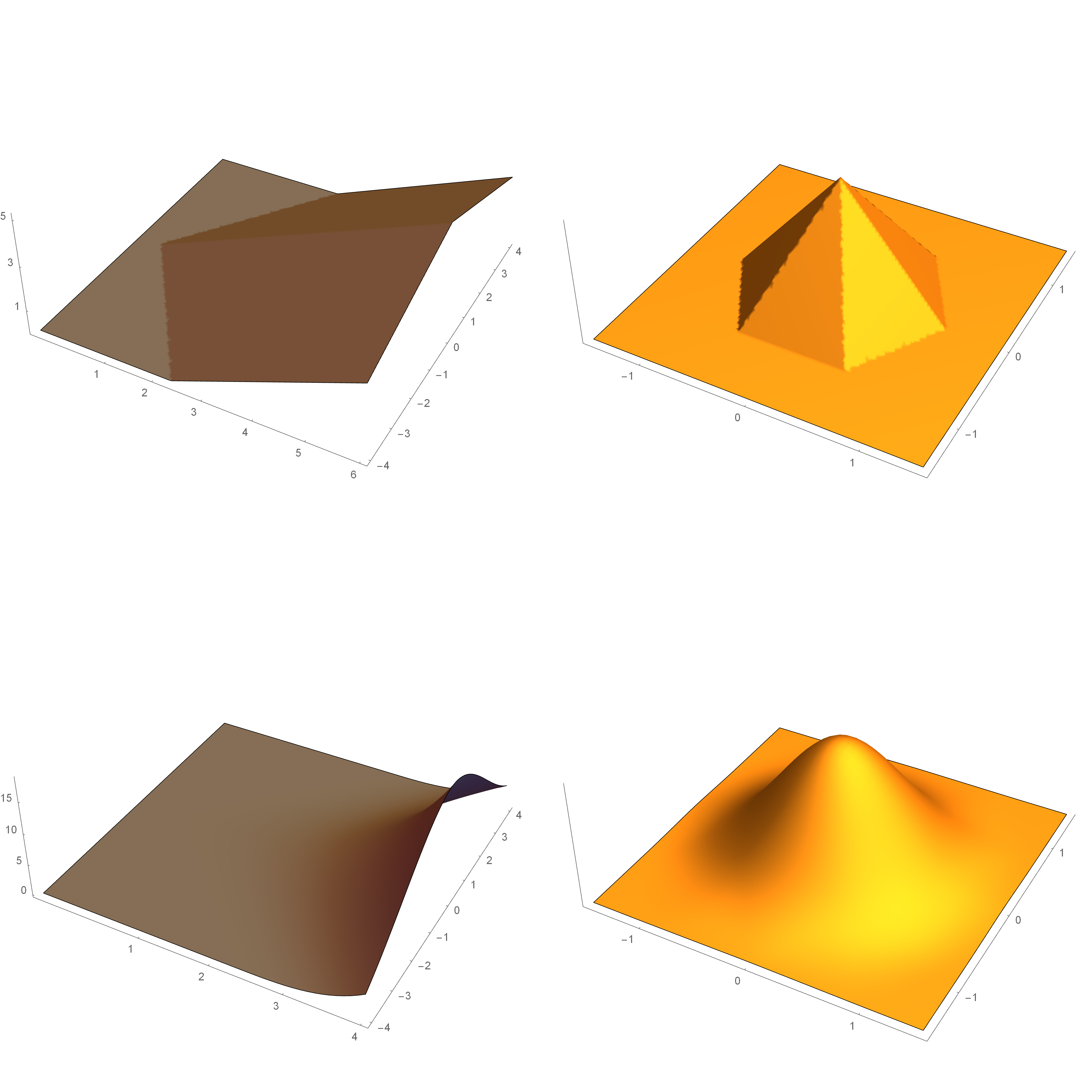}}
    \end{center}
    \caption{$\Cnspline$ generator for $\Bnspline$ for $n=1$ (top row) and $n=2$ (bottom row).  The functions in the top row are piecewise linear and the functions on the bottom row are piecewise quartic.  The support sets for these splines are plotted (first and last) in Figure \ref{fig:hexsplinesupport}.}
    \label{fig:condatsplines}
\end{figure}

It is straightforward to generalize \eqref{condatcoeffs} to general $\mathbf{n} = \left(n_1, n_2, n_3\right)$:
\begin{generalcoeffs}
For $\bX = \left\{ \bx^1, \bx^2, \bx^3\right\}$ where $\bx^3 = \bx^1 + \bx^2 = (1,0)^T$, $M = \left[\bx^1 \bx^2\right]$ and $\mathbf{n} = \left(n_1, n_2, n_3\right) \in \mathbb{Z}^3_+$, \eqref{expproduct} can be expressed as
\begin{equation*}
    \hat{f}(\omega) = \sum_{\mathbf{k}\in\mathbb{Z}^2} c_{k_1,k_2} e^{-i\langle \bomega, M\mathbf{k}\rangle}
\end{equation*}
where
\begin{align*}
    c_{k_1,k_2} &= \sum_{m=\max\left(n_3-n_1+k_1, n_3-n_2+k_2,0\right)}^{n_3 + \min\left(k_1,k_2,0\right)} \binom{n_1}{n_1-n_3-k_1+m}
        \binom{n_2}{n_2-n_3-k_2+m}\binom{n_3}{m} (-1)^{k_1+k_2+m}.
\end{align*}
        \label{prop:generalcoeffs}
\end{generalcoeffs}
\begin{proof}
We follow the argument used by Condat and Van De Ville \cite{article:condat}.  For the sake of notation, set $z_1 = e^{-i\langle \bomega, \bx^1\rangle}$ and $z_2 = e^{-i\langle \bomega, \bx^2\rangle}$.  Then \eqref{expproduct} becomes
\begin{equation*}
    f(\bomega) = \left(1-z_1^{-1}\right)^{n_1} \left( 1 - z_2^{-1}\right)^{n_2} \left(z_1 z_2 - 1\right)^{n_3}
\end{equation*}
Expanding each factor gives
\begin{equation*}
    f(\bomega) = \sum_{j=0}^{n_1}\sum_{k=0}^{n_2}\sum_{m=0}^{n_3} \binom{n_1}{j} \binom{n_2}{k} \binom{n_3}{m} (-1)^{j+k+m} z_1^{n_3-m-j} z_2^{n_3-m-k}.
\end{equation*}
Making the substitution $k_1 = j+m-n_3$ and $k_2 = k+m-n_3$ and simplifying yields
\begin{equation*}
    f(\bomega) = \sum_{m=0}^{n_3} \sum_{k_1 = m-n_3}^{n_1+m-n_3} \sum_{k_2=m-n_3}^{n_2+m-n_3} \binom{n_1}{n_1-n_3+m-k_1} \binom{n_2}{n_2-n_3+m-k_2}\binom{n_3}{m} (-1)^{k_1+k_2+m} z_1^{-k_1}z_2^{-k_2}.
\end{equation*}
We can interchange the order of summations by changing the limits on the inner two sums to $k_1, k_2 \in \mathbb{Z}$ and restricting the limits on $m$ so that the binomial coefficient factors are not zero.  These changes give the desired result for the $c_{k_1,k_2}$.
Finally, note that
\begin{equation}
    z_1^{-k_1}z_2^{-k_2} = e^{-i\left(\langle \bomega, \bx^1\rangle k_1 + \langle \bomega, \bx^2\rangle k_2\right)}
        = e^{-i\left(M^T \bomega\right)^T \cdot \mathbf{k}}
        = e^{-i\bomega^T \cdot M\mathbf{k}}
        = e^{-i\langle \bomega, M\mathbf{k}\rangle}.        \label{zprod}
\end{equation}
\end{proof}
We can use Propositions \ref{prop:cone3direction} and \ref{prop:generalcoeffs} in conjunction with \eqref{boxsum} to find explicit representations of hex splines.  The hex splines corresponding to the second and third knot sets from Figure \ref{fig:hexsplinesupport} are plotted in Figure \ref{fig:hexsplines}.
\begin{figure}[ht]
    \begin{center}
        \resizebox{4in}{!}{\includegraphics{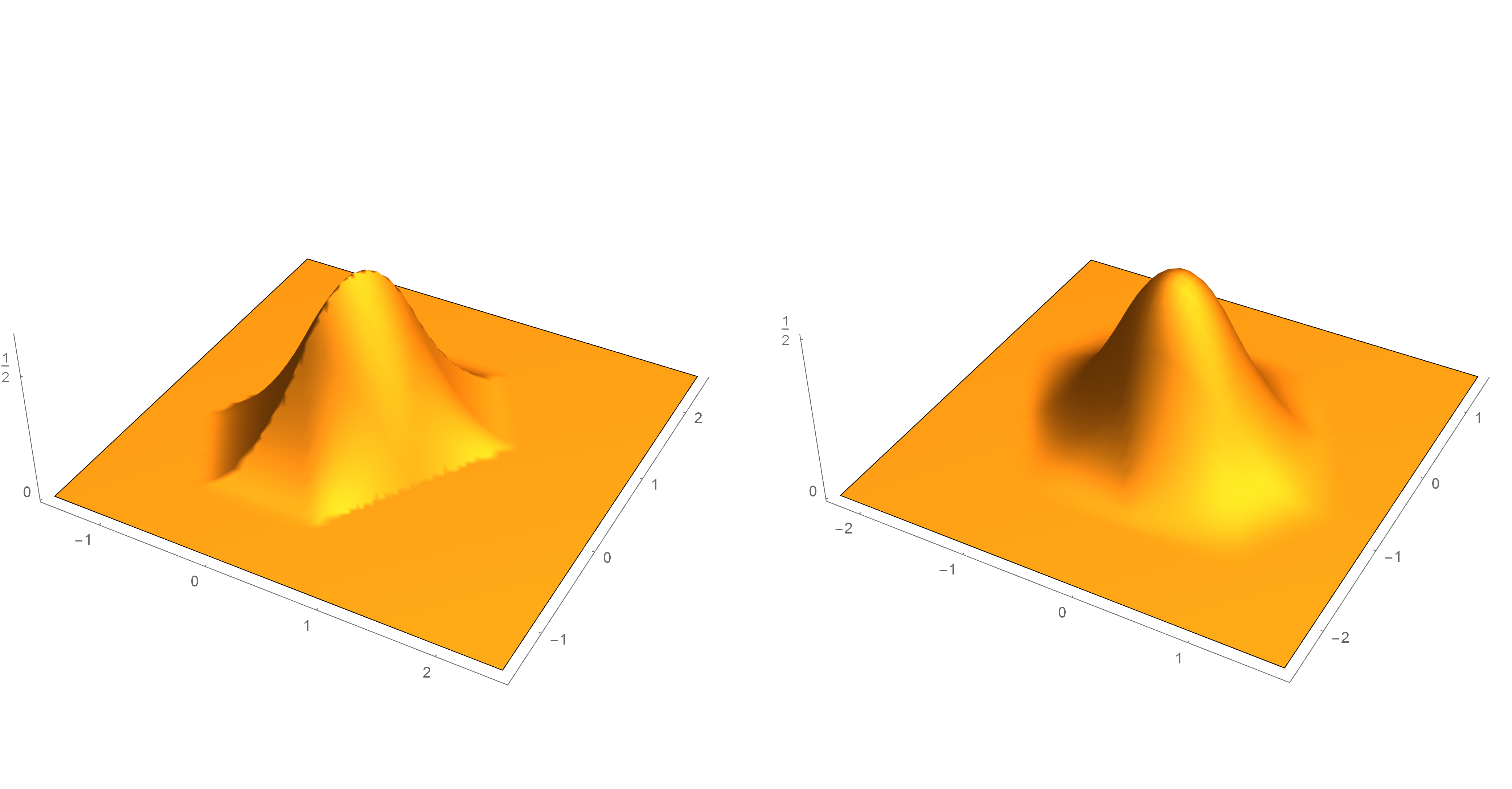}}
    \end{center}
    \caption{The hex splines $B_{(2,1,1)}(\bx \vb \bXm)$ and $B_{(1,2,1)}(\bx \vb \bXm)$.  The knots are $\bx^1 = \left(\frac{1}{2}, -\frac{\sqrt{3}}{2}\right)^T$, $\bx^2 = \left(\frac{1}{2}, \frac{\sqrt{3}}{2}\right)^T$, and $-\bx^3 = -\bx^1-\bx^2 = (-1,0)^T$.}
    \label{fig:hexsplines}
\end{figure}
\section{Fractional Cone Splines and Hex Splines} \label{section:fractionalconeshexes}
We can generalize the cone splines and hex splines from Sections \ref{section:conesplines} and \ref{section:hexsplines} by replacing the integer ``weights" that represent how many times we repeat the knots in $\bX$ with positive real-valued weights.
\subsection{Fractional cone splines}
We start with the fractional cone spline.
\begin{fracconedefn}\label{fracconedefn}
Let $\bX = \left\{ \bx^1, \ldots, \bx^n\right\} \subset \mathbb{R}^s \backslash \left\{ \mathbf{0}\right\}$, $n \geq s$, with $\spn(\bX) =\mathbb{R}^s$ and assume $\balpha = \left( \alpha_1,\ldots, \alpha_n\right) \in \mathbb{R}_+^n := \left\{ \bx \in \mathbb{R}^n \vb x_k > 0\right\}$.  We define the {\em fractional cone spline} in terms of its Fourier transform as the tempered distribution given by
\begin{equation}
    \widehat{\Calpha}(\bomega \vb \bX) =  \prod_{k=1}^n \left( \frac{1}{i\langle \bomega, \bx^k\rangle} + \sigma_k \pi \delta\left(\langle \bomega, \bx^k\rangle\right)\right)^{\alpha_k} \label{fractionalcone}
\end{equation}
where
\begin{equation*}
    \sigma_k = \left\{ \begin{array}{ll} 0,& \alpha_k \notin \Z_+ \\ 1, & \alpha_k \in \Z_+\end{array}\right. .
\end{equation*}
\label{defn:fracconedefn}
\end{fracconedefn}
Our first proposal gives a closed form for $\Calphaspline$ when $n=s$.
\begin{fraccones}
    \label{prop:fraccones}
Let $\bX = \left\{ \bx^1,\ldots, \bx^s\right\}$ be a linearly independent set of vectors in $\mathbb{R}^s$ with $M = \left[ \bx^1 \cdots \bx^s\right]$, $\balpha = \left(\alpha_1,\ldots,\alpha_s\right)\in \mathbb{R}_+^s$, and $\left\{ \bxd^1,\ldots,\bxd^s\right\}$ the dual basis to $\bX$.  Then
\begin{equation}
    \Calphaspline = \frac{1}{\lvert \det(M)\rvert} \prod_{k=1}^s  \frac{\langle \bx, \bxd^k\rangle_+^{\alpha_k-1}}{\Gamma\left(\alpha_k\right)}.  \label{fraccones2}
\end{equation}
\end{fraccones}
\begin{proof}
Consider the canonical basis $\bE = \left\{ \be^1,\ldots,\be^s\right\}$ for $\mathbf{R}^s$.  Using \eqref{fractionalcone}, we see that
\begin{equation*}
    \widehat{\Calpha}(\bomega \vb \bE) = \prod_{k=1}^s \left( \frac{1}{i\omega_k} + \sigma_k \pi \delta\left(\omega_k\right)\right)^{\alpha_k}.
\end{equation*}
Thus
\begin{equation*}
    \CalphasplineE = \prod_{k=1}^s \int_{\mathbb{R}} \left( \frac{1}{i\omega} + \sigma_k \pi \delta(\omega)\right)^{\alpha_k} e^{i\omega x_k} \diff{\omega}.
\end{equation*}
When $\sigma_k = 1$, $\alpha_k \in \mathbb{Z}$ and the corresponding factor is simply an $\alpha_k$-fold convolution $\left(x_k\right)_+^0 * \cdots \left(x_k\right)_+^0 = \frac{1}{\Gamma\left(\alpha_k\right)} \left(x_k\right)_+^{\alpha_k-1}$.  When $\sigma_k = 0$, $\alpha_k \notin \mathbb{Z}$ and the factor is the inverse Fourier transform of $\frac{1}{(i\omega)^{\alpha_k}}$ which is again $\frac{1}{\Gamma\left(\alpha_k\right)} \left(x_k\right)_+^{\alpha_k-1}$.  Thus we see that
\begin{equation}
    \CalphasplineE = \prod_{k=1}^s \frac{\left(x_k\right)_+^{\alpha_k-1}}{\Gamma\left(\alpha_k\right)}.      \label{fracconecanonical}
\end{equation}
Note that
\begin{equation}
    \widehat{\Calpha}(\bomega \vb \bX) = \widehat{\Calpha}\left( M^T \bomega \vb \bE\right). \label{altCalphaE}
\end{equation}
We use the fact that the inverse Fourier transform for $\frac{1}{\vert \det(A)\rvert}\hat{f}(A\bomega)$, $A$ nonsingular, is $f(A\bx)$ in conjunction with \eqref{fracconecanonical} and \eqref{altCalphaE} to infer that
\begin{equation*}
    \Calphaspline = \lvert \det\left(M^{-1}\right)\rvert \Calpha\left(M^{-1}\bx \vb \bE\right)
        = \frac{1}{\lvert \det(M)\rvert} \prod_{k=1}^s  \frac{\langle \bx, \bxd^k\rangle_+^{\alpha_k-1}}{\Gamma\left(\alpha_k\right)}.
\end{equation*}
\end{proof}
Figure \ref{fig:fraccones2} illustrates the results of Proposition \ref{prop:fraccones}.
\begin{figure}[ht]
    \begin{center}
        \resizebox{\textwidth}{!}{\includegraphics{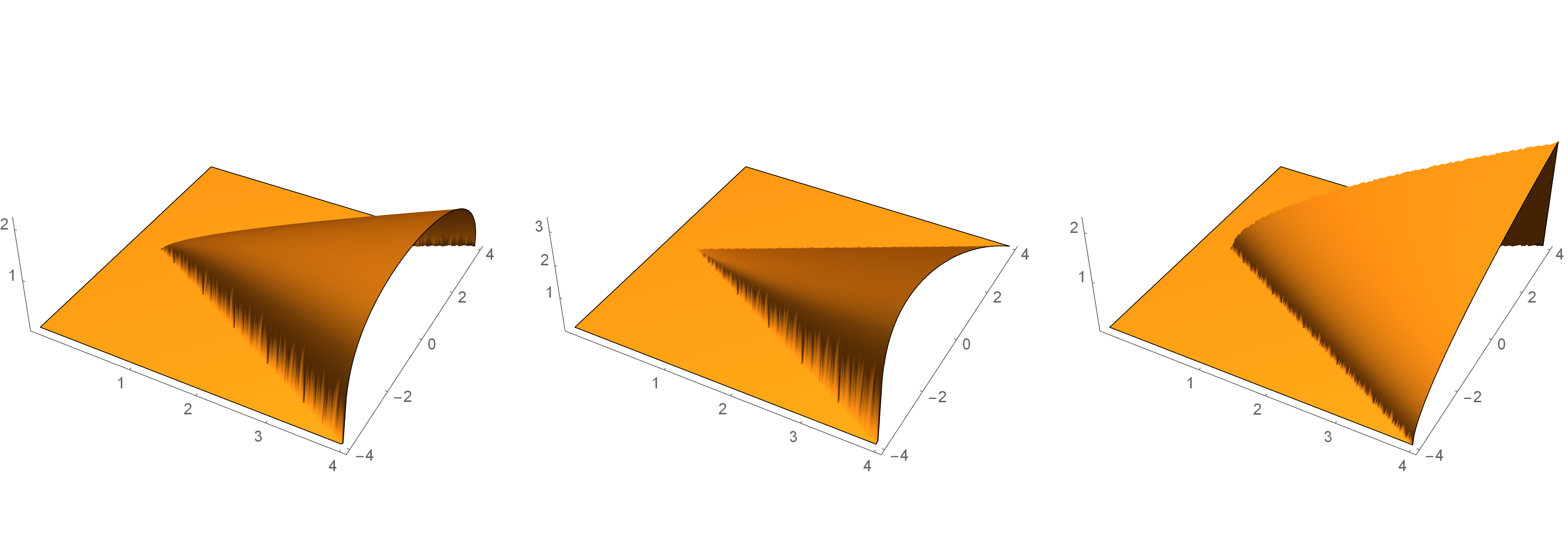}}
    \end{center}
    \caption{The splines $\Calphaspline$ for $\bX = \left\{ (1,1)^T, (1,-1)^T\right\}$ and, left to right, $\balpha = \left(\frac{5}{4}, \frac{3}{2}\right)$, $\balpha = \left(\frac{6}{5}, 2\right)$, and $\balpha = \left( \frac{3}{2}, 1\right)$.}
    \label{fig:fraccones2}
\end{figure}
For $\Calphaspline$ given by \eqref{fraccones2}, note that $\supp(\Calpha) = \cone(\bX)$ and it is straightforward to show that these $\Calphaspline$ obey a recurrence formula similar to \eqref{conerecurrence}.
\begin{fracconerecurrence}
    \label{prop:fracconerecurrence}
Let $\bX = \left\{ \bx^1, \ldots, \bx^s\right\}$ be a linearly independent set of vectors in $\mathbb{R}^s$ and $\be^k$ the $k$th canonical vector in $\mathbb{R}^s$.  If $\balpha - \be^k \in \mathbb{R}_+^s$ and $\lvert \balpha \rvert = \alpha_1 + \cdots + \alpha_s \neq s$, then
\begin{equation*}
    \Calphaspline = \frac{1}{\lvert \balpha\rvert - s} \sum_{k=1}^s \langle \bx, \bx^k\rangle C_{\balpha - \be^k}(\bx \vb \bX).
\end{equation*}
\end{fracconerecurrence}
\begin{proof}
For $k=1,\ldots,s$,
\begin{align*}
    \langle \bx, \bxd^k\rangle C_{\balpha - \be^k}(\bx \vb \bX) &= \langle \bx, \bxd^k\rangle \left( \prod_{\substack{j=1 \\ j \neq k}}^s \frac{\langle \bx, \bxd^j\rangle_+^{\alpha_j-1}}{\Gamma\left(\alpha_j\right)}\right) \frac{\langle \bx, \bxd^k\rangle_+^{\alpha_k-2}}{\Gamma\left(\alpha_k-1\right)} \\
    &= \left( \prod_{j=1}^s \frac{\langle \bx, \bxd^j\rangle_+^{\alpha_j-1}}{\Gamma\left(\alpha_j\right)}\right) \left(\alpha_k-1\right) \\
    &= \left(\alpha_k-1\right)\Calphaspline.
\end{align*}
Summing both sides for $k=1,\ldots, s$, gives the desired result.
\end{proof}
We can generalize Proposition \ref{prop:cone3direction} for fractional cone splines on $3$-directional meshes.
\begin{propfraccone3}
    \label{prop:fraccone3}
Let $\bX$ be a $3$-directional mesh in $\mathbb{R}^2$ with $\bx^3 = \bx^1+\bx^2 = (1,0)^T$ and $\balpha \in \mathbb{R}_+^3$.  Then
\begin{equation*}
    \Calphaspline = \frac{1}{x^1_2} \left\{ \begin{array}{ll} \frac{1}{\Gamma\left(\alpha_1\right)\Gamma\left(\alpha_2+\alpha_3\right)} \left(\frac{x_2}{x^1_2}\right)^{\alpha_1-1} \left(x_1 - \frac{x^1_1}{x^1_2}x_2\right)_+^{\alpha_2+\alpha_3-1}{}_2F_1\left(\scz{1-\alpha_1}, \scz{\alpha_2}; \scz{\alpha_2 + \alpha_3}; -\frac{x_1-\frac{x^1_1}{x^1_2}x_2}{\frac{x_2}{x^1_2}}\right), & x_2 > 0 \\
    \frac{1}{\Gamma\left(\alpha_2\right)\Gamma\left(\alpha_1+\alpha_3\right)} \left(-\frac{x_2}{x^1_2}\right)^{\alpha_2-1} \left(x_1 - \frac{x^1_1-1}{x^1_2}x_2\right)_+^{\alpha_1+\alpha_3-1}{}_2F_1\left(\scz{1-\alpha_2}, \scz{\alpha_1}; \scz{\alpha_1 + \alpha_3}; \frac{x_1-\frac{x^1_1-1}{x^1_2}x_2}{\frac{x_2}{x^1_2}}\right), & x_2 < 0 \\
    \frac{1}{\Gamma\left(\alpha_1\right)\Gamma\left(\alpha_2\right)}\left(\sum\limits_{k=0}^\infty \frac{(-1)^k}{k! \Gamma\left(\alpha_3-k\right)\left(\alpha_1+\alpha_2+k-1\right)}\right)\left(x_1\right)_+^{\alpha_1+\alpha_2+\alpha_3-2}, & x_2 = 0
        \end{array}\right.
\end{equation*}
\end{propfraccone3}
\begin{proof}
This proof follows in much the same was as that of Proposition \ref{prop:cone3direction}.  We can write
\begin{equation*}
    \widehat{\Calpha}(\bomega \vb \bX) = \widehat{C}_{\left(\alpha_1,\alpha_2\right)}\left(\bomega \vb \left\{ \bx^1, \bx^2\right\}\right)\cdot \left( \frac{1}{i\omega_1}+\pi \sigma_3 \delta\left(\omega_1\right)\right),
\end{equation*}
where $\sigma_3 = 1$ if $\alpha_3 \in \mathbb{Z}$, $\sigma_3 = 0$ if $\alpha_3 \notin\mathbb{Z}$, so that
\begin{align}
    \Calphaspline &= C_{\left( \alpha_1,\alpha_2 \right)}(\bomega \vb \bX) * \frac{\left(x_1\right)_+^{\alpha_3-1}}{\Gamma\left( \alpha_3\right)} \delta\left(x_2\right) \nonumber \\
        &= \left(\lvert\det(M)\rvert\prod_{k=1}^3 \Gamma\left(\alpha_k\right)\right)^{-1} \int_{-\infty}^{x_1} \langle \left(u,x_2\right), \bxd^1\rangle_+^{\alpha_1-1} \langle \left(u,x_2\right), \bxd^2\rangle_+^{\alpha_2-1} \left(x_1 - u\right)^{\alpha_3-1}\diff{u}. \label{calphaconvolve}
\end{align}
The dual basis vectors are given in \eqref{dualvectors} so that \eqref{calphaconvolve} becomes
\begin{equation}
    \Calphaspline = \left(x_2^1 \prod_{k=1}^3 \Gamma\left(\alpha_k\right)\right)^{-1}  \int_{-\infty}^{x_1} \left( u - \left(\frac{x_1^1-1}{x^1_2}\right) x_2\right)_+^{\alpha_1-1}\left( u -\frac{x_1^1}{x_2^1} x_2\right)_+^{\alpha_2-1}\left(x_1-u\right)^{\alpha_3-1}\diff{u}.
                    \label{alphaintegral}
\end{equation}
If $x_2 > 0$, we have $\frac{x_1^1-1}{x^1_2}x_2 < \frac{x^1_1}{x^1_2}x_2 < x_1$ and \eqref{alphaintegral} becomes
\begin{equation*}
   \Calphaspline = \left(x_2^1 \prod_{k=1}^3 \Gamma\left(\alpha_k\right)\right)^{-1}\int\limits_{\frac{x_1^1}{x_2^1} x_2}^{x_1} \left( u - \left(\frac{x_1^1-1}{x^1_2}\right)x_2\right)^{\alpha_1-1}\left( u - \frac{x_1^1}{x_2^1} x_2\right)^{\alpha_2-1}\left(x_1-u\right)^{\alpha_3-1}\diff{u}
\end{equation*}
and we can use \eqref{stamko} to complete the proof.  The case where $x_2 < 0$ is similar.  For $x_2=0$, we use \eqref{alphaintegral} to write
\begin{equation*}
    \Calpha\left(\left(x_1,0\right)\vb \bX\right) = \left(x_2^1 \prod_{k=1}^3 \Gamma\left(\alpha_k\right)\right)^{-1}  \int_0^{x_1} u^{\alpha_1+\alpha_2-2}\left(x_1 - u\right)^{\alpha_3-1}\diff{u}.
\end{equation*}
The generalized binomial theorem can be used to complete the proof.
\end{proof}
\begin{reductionremark}
If each of $\alpha_1, \alpha_2, \alpha_3$ are integers, Proposition \ref{prop:fraccone3} reduces to Proposition \ref{prop:cone3direction}.
\end{reductionremark}
Figure \ref{fig:fraccone3} illustrates the results of Proposition \ref{prop:fraccone3}.
\begin{figure}[ht]
    \begin{center}
        \resizebox{\textwidth}{!}{\includegraphics{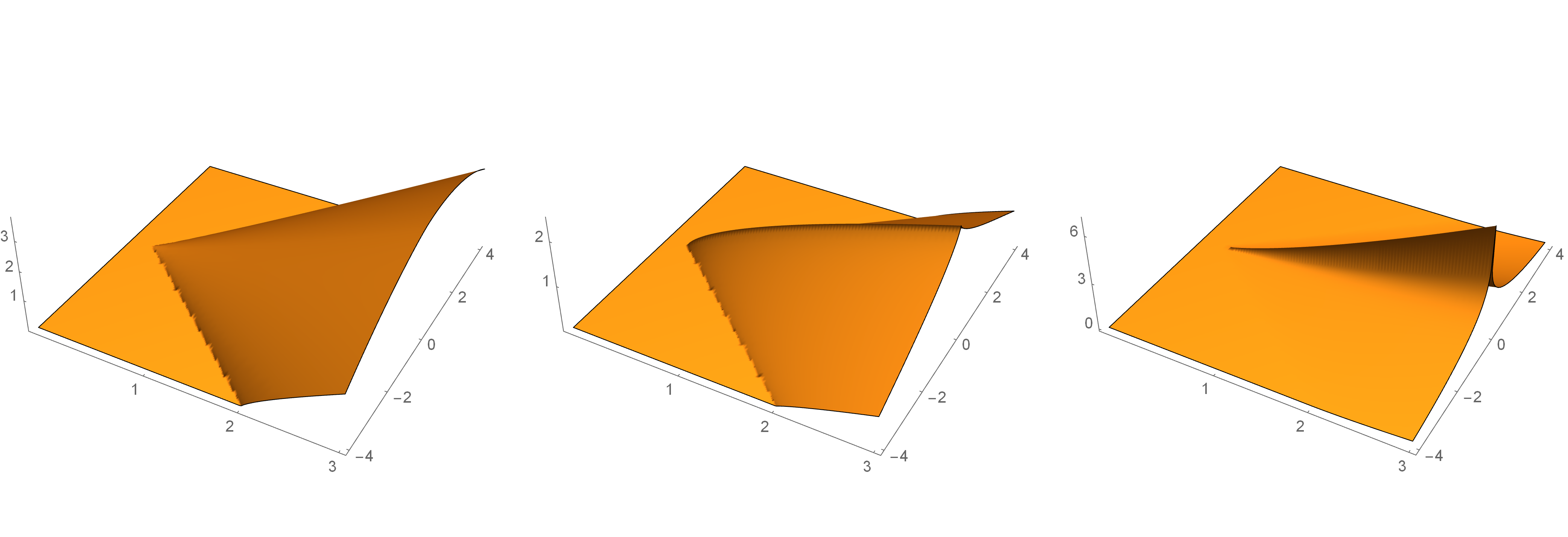}}
    \end{center}
    \caption{The splines $\Calphaspline$ for $\bX = \left\{ \left(\frac{1}{2},1\right)^T, \left(-\frac{1}{2},-1\right)^T, (1,0)^T \right\}$ and, left to right, $\balpha = \left(\frac{3}{2}, \frac{5}{4}, \frac{1}{3} \right)$, $\balpha = \left(\frac{9}{10}, \frac{4}{5}, 1\right)$, and $\balpha = \left( \frac{3}{5}, \frac{3}{5}, 2\right)$.}
    \label{fig:fraccone3}
\end{figure}
\subsection{Fractional hex splines}
In this subsection, we define a bivariate fractional hex spline implicitly in terms of its Fourier transform.  We then use this definition to derive an explicit formula for the spline.  Motivated by \eqref{boxftviaconeft} and Definition \ref{defn:fracconedefn}, we make the following definition:
\begin{frachexdefn}\label{frachexdefn}
Let $\bX = \left\{ \bx^1, \bx^2, \bx^3 \right\}$ be a $3$-directional mesh in $\mathbb{R}^2$ with $\bx^1, \bx^2$ linearly independent and $\bx^3 = \bx^1 + \bx^2 = (1,0)^T$ and assume $\balpha \in \mathbb{R}_+^3$.  We define the \emph{fractional hex spline of order $\balpha$} in terms of its Fourier transform (as the tempered distribution):
\begin{equation}
    \widehat{\Balpha}(\bomega \vb \bXm) =  \left( 1 - e^{-i\langle \bomega, \bx^1\rangle}\right)^{\alpha_1}\left( 1 - e^{-i\langle \bomega, \bx^2\rangle}\right)^{\alpha_2}  \left( e^{i\langle \bomega, \bx^3\rangle} - 1\right)^{\alpha_3} \widehat{\Calpha}(\bomega \vb \bX) \label{fractionalhexft}
\end{equation}
where $\bXm = \left\{ \bx^1, \bx^2, -\bx^3\right\}$.
\label{defn:frachexdefn}
\end{frachexdefn}
Taking inverse Fourier transforms of \eqref{fractionalhexft} gives
\begin{equation}
    \Balpha(\bomega \vb \bXm) = f(\bx) * \Calpha(\bx \vb \bX)       \label{hexalphaasconv}
\end{equation}
where
\begin{equation*}
    \hat{f}(\bomega) = \left( 1 - e^{-i\langle \bomega, \bx^1\rangle}\right)^{\alpha_1}\left( 1 - e^{-i\langle \bomega, \bx^2\rangle}\right)^{\alpha_2}  \left( e^{i\langle \bomega, \bx^3\rangle} - 1\right)^{\alpha_3}.
\end{equation*}
Letting $z_k = e^{-i\langle\bomega, \bx^k\rangle}$, $k=1,2$, we use the generalized binomial theorem to write
\begin{align}
    \hat{f}(\bomega) &= \sum_{\ell \geq 0} \sum_{k_1 \geq 0} \sum_{k_2 \geq 0} \binom{\alpha_1}{k_1} \binom{\alpha_1}{k_2} \binom{\alpha_1}{\ell} (-1)^{k_1+k+\ell} z_1^{-\ell-k_1}z_2^{-\ell-k_2}\left(z_1z_2\right)^{\alpha_3} \nonumber \\
        &= \sum_{\ell \geq 0} \sum_{k_1 \geq -\ell} \sum_{k_2 \geq -\ell} \binom{\alpha_1}{k_1+\ell} \binom{\alpha_1}{k_2+\ell} \binom{\alpha_1}{\ell} (-1)^{k_1+k+\ell}z_1^{-k_1}z_2^{-k_2}\left(z_1z_2\right)^{\alpha_3} \nonumber \\
        &= \sum_{\mathbf{k}\in\mathbb{Z}^2} \left( \sum_{\ell \geq \max\left(k_1,k_2,0\right)} \binom{\alpha_1}{\ell-k_1}\binom{\alpha_2}{\ell-k_2}\binom{\alpha_3}{\ell} (-1)^{k_1+k+\ell}\right)z_1^{-k_1}z_2^{-k_2}\left(z_1z_2\right)^{\alpha_3} \nonumber \\
        &= \sum_{\mathbf{k}\in\mathbb{Z}^2} c_{k_1,k_2} e^{-i\langle\bomega, M\left(\mathbf{k}-\boldsymbol{\alpha}_3\right)\rangle} \label{hatfhex}
\end{align}
where $\mathbf{k} = \left(k_1,k_2\right)^T$, $\boldsymbol{\alpha}_3 = \left(\alpha_3,\alpha_3\right)^T$, $M = \left[ \bx^1 \ \bx^2 \right]$ and
\begin{equation}
    c_{k_1,k_2} = \sum_{\ell \geq \max\left(k_1,k_2,0\right)} \binom{\alpha_1}{\ell-k_1}\binom{\alpha_2}{\ell-k_2}\binom{\alpha_3}{\ell} (-1)^{k_1+k_2+\ell}.         \label{ckforhexf}
\end{equation}
We have used \eqref{zprod} to simplify the last identity \eqref{hatfhex} above.

We take inverse Fourier transforms of \eqref{hatfhex} to obtain
\begin{equation*}
    f(\bx) = \sum_{\mathbf{k}\in\mathbb{Z}^2} c_{k_1,k_2} \delta\left( \bx - M\left(\mathbf{k}-\boldsymbol{\alpha}_3\right)\right).
\end{equation*}
The relation \eqref{hexalphaasconv} and the above argument gives rise to the following time-domain representation of $\Balphasplinem$.
\begin{propfrachex3}
    \label{prop:frachex3}
For the fractional hex spline in Definition \eqref{defn:frachexdefn}, we have the following explicit representation:
\begin{equation}
    \Balphasplinem = \sum_{\mathbf{k}\in\mathbb{Z}^2} c_{k_1,k_2} \Calpha\left(\bx- M\left(\mathbf{k}-\boldsymbol{\alpha}_3\right) \vb \bX\right)        \label{hexexplicit}
\end{equation}
where the coefficients $c_{k_1,k_2}$ are given by \eqref{ckforhexf} and $\Calphaspline$ is given by the relation in Proposition \ref{prop:fraccone3}.
\end{propfrachex3}
Proposition \ref{prop:frachex3} is used to evaluate the hex spline plotted in Figure \ref{fig:frachex3}.
\begin{figure}[ht]
    \begin{center}
        \resizebox{\textwidth}{!}{\includegraphics{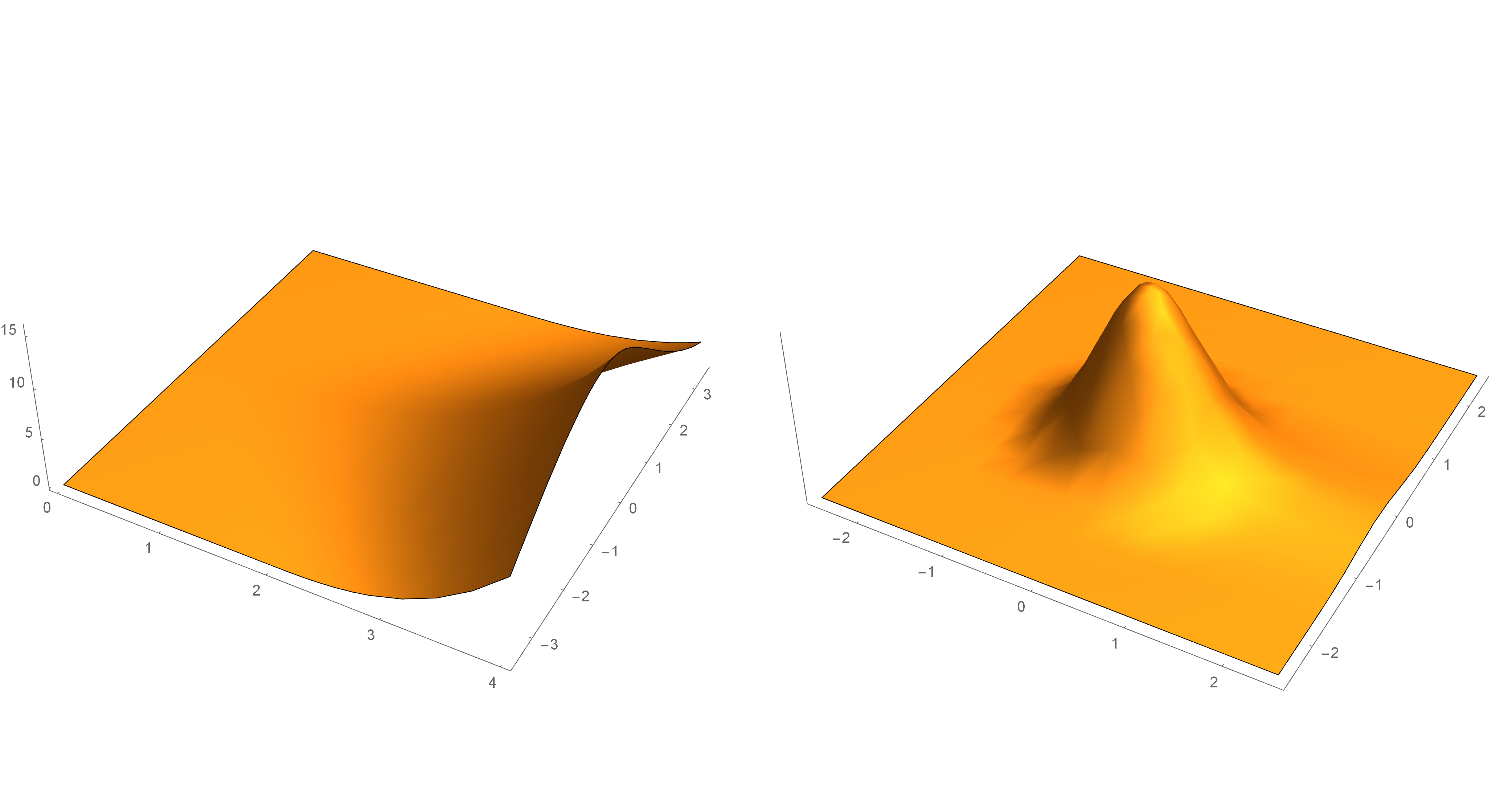}}
    \end{center}
    \caption{The fractional cone spline $\Calphaspline$ and box spline $\Balphasplinem$ for $\bX = \left\{ \left(\frac{1}{2}, \frac{\sqrt{3}}{2}\right)^T, \left(\frac{1}{2}, -\frac{\sqrt{3}}{2}\right)^T, (1,0)^T \right\}$ and $\balpha = \left(1.5, 2.1, 2\right)$.}
    \label{fig:frachex3}
\end{figure}
\begin{differenceremark}
Notice that \eqref{fractionalhexft} can also be written in the form
\begin{equation*}
    \widehat{\Balpha}(\bomega \vb \bXm) = \widehat{\nabla}_\bX^\alpha\,\widehat{\Calpha}(\bomega \vb \bX),
\end{equation*}
where $\widehat{\nabla}_\bX^\alpha$ is now the Fourier transform of the {\em fractional forward difference operator  of order $\alpha := (\alpha_1,\alpha_2, \alpha_3)\in \R_+^3$} defined by
\begin{equation*}
    \nabla_\bX^\alpha := \prod_{i=1}^3\,\nabla_{\bx^i}^{\alpha_i},
\end{equation*}
where, for any function $g:\R^3\to \R$,
\begin{equation*}
    \nabla_{\bx^i}^{\alpha_i} g := \sum_{k=0}^\infty \binom{\alpha_i}{k} (-1)^k g(\bullet - k \bx^i),
\end{equation*}
(with a slight modification for $i = 3$).
\end{differenceremark}

\section{Hex Spline Properties and Splines of Complex Order}
We conclude the paper by stating several properties obeyed by hex splines and developing definitions of cone and hex splines of complex order.  Unlike polynomial box splines, fractional hex splines are not compactly supported for $\balpha \notin \mathbb{Z}^3$. Thus we give decay rates for $\Balphasplinem$.  We also show that hex splines are refinable and give an explicit formula for the refinement mask.  Finally we show that a hex spline and its translates along the integer lattice $\mathbb{Z}^2$ form a Riesz sequence and thus a Riesz basis for an appropriate shift-invariant subspace of $L^2\left(\mathbb{R}^2\right)$.

\subsection{Decay and Refinement}
We first investigate the decay rate of the fractional hex spline $\Balphasplineb$. To this end, recall that a function $g: \R^2\to \R$ belongs to the Sobolev space $W^{r,p}(\R^2)$, $r > 0$, $1< p < \infty$, iff
\beq\label{sobolev}
(1 + \norm{\bullet}^2)^{r/2}\,\widehat{g} \in L^p (\R^2).
\eeq
It follows from Definition \ref{frachexdefn} that $|\widehat{\Balpha}(\bomega \vb \bXm)| \leq \frac{K}{\norm{\bomega}^{|\balpha|}}$, where $K$ denotes a positive constant that depends only on $\bX$.  Hence, $\Balphasplineb\in L^2(\R^2)$ if $|\balpha| > 1$, and $\Balphasplineb\in L^1(\R^2)\cap L^2(\R^2)$ if $|\balpha| > 2$. Equation \eqref{sobolev} together with the above estimate on $|\widehat{\Balpha}(\bullet \vb \bXm)|$ imply that
\[
\Balphasplineb\in W^{r,p}(\R^2),\quad\text{for $r < |\balpha| - \frac2p$}.
\]
At this point, we recall the Sobolev Embedding Theorem in $\R^2$ \cite[Theorem 4.12]{book:adams}:
If $j\in \N_0$, $1\leq m\in \N$, and $mp > 2$, then
\[
W^{j+m,p}(\R^2) \hookrightarrow C^j_b (\R^2),
\]
where $C_b^j(\R^2)$ denoted the space of all bounded continuous functions with derivatives up to order $j$.

Choose $m=2$ and $p=2$. Then $\widehat{\Balpha}(\bullet \vb \bXm)\in C_b^j(\R^2)$ for $|\balpha| > j + 3$. In other words, for every multi-index $\bmu$ with $|\bmu|\leq j$, $\partial^\bmu\widehat{\Balpha}(\bullet \vb \bXm)$ is bounded and continuous. But then $\bx^\bmu {\Balpha}(\bx \vb \bXm)$ is uniformly continuous and vanishes at infinity. Hence,
\[
{\Balpha}(\bx \vb \bXm) \in \mathcal{O}(\bx^{-\bmu}), \quad\text{for $|\balpha| > j + 3$}.
\]
The fractional hex splines $\Balphasplineb$ satisfy two-scale relations of the form
\beq\label{twoscale}
{\Balpha}(\bx \vb \bXm) = \sum_{\bk\in \Z^2} h(\bk) {\Balpha}(2\bx - M(\bk-\balpha_3) \vb \bXm),
\eeq
with $M = \left[ \bx^1 \ \bx^2 \right]$ and $\boldsymbol{\alpha}_3 = \left(\alpha_3,\alpha_3\right)^T$, which are valid for $|\balpha| > 1$ and for almost all $\bx\in \R^2$.

To show the validity of the above refinement equation, recall that -- if such an equation holds -- the coefficients $h(\bk)$ are the Fourier coefficients of the frequency response of the refinement filter $H_{\balpha}$. In other words,
\begin{align*}
H_{\balpha} (\bomega) &= \frac{2 \widehat{\Balpha}(2 \bomega \vb \bXm)}{\widehat{\Balpha}(\bomega \vb \bXm)} = 2\,\prod_{j=1}^3\left(\frac{i\inn{\bomega}{\bx^j}}{i\inn{2\bomega}{\bx^j}}\right)^{\alpha_j}
\left(\frac{1 - e^{-2i\inn{\bomega}{\bx^j}}}{1 - e^{-i\inn{\bomega}{\bx^j}}}\right)^{\alpha_j}\quad\text{(with modification for $j=3$)}\\
&= \frac{1}{2^{|\balpha|-1}} \prod_{j=1}^2 \left(1+ e^{-i\inn{\bomega}{\bx^j}}\right)^{\alpha_j} \left(1+ e^{i\inn{\bomega}{\bx^3}}\right)^{\alpha_3},\quad \text{(almost everywhere)}
\end{align*}
with $\bx^3 = \bx^1 + \bx^2 = (1,0)^T$. Expanding the threefold product in the last equation above using the generalized binomial theorem and proceeding as in the derivation of \eqref{hatfhex}, we obtain
\beq\label{H}
H_{\balpha} (\bomega) = \frac{1}{2^{|\balpha|-1}} \sum_{\mathbf{k}\in\mathbb{Z}^2} a (\bk) e^{-i\langle\bomega, M\left(\mathbf{k}-\boldsymbol{\alpha}_3\right)\rangle},
\eeq
where
\[
a (\bk) := \sum_{\ell \geq \max\left(k_1,k_2,0\right)} \binom{\alpha_1}{\ell-k_1}\binom{\alpha_2}{\ell-k_2}\binom{\alpha_3}{\ell}.
\]
Equation \eqref{H} now yields the desired two-scale relation \eqref{twoscale}.
\subsection{Riesz basis property}
We next consider bivariate hex splines and their translates along the lattice $\mathbb{Z}^2$ and show that this set of functions forms a Riesz sequence and thus a Riesz basis for the appropriate shift-invariant subspace of $L^2(\R^2)$. To this end, we introduce the following fractional hex spline space.

For the knot set $\bXm = \left\{ \bx^1, \bx^2, -\bx^3\right\}\subset \mathbb{R}^2$, where $\bx^3 = \bx^1 + \bx^2$, and a multi-index $\balpha\in \R_+^3$, let
\beq\label{S}
\cS^{\balpha} := \cS^{\balpha} (B_{\balpha}(\bullet \vb \bXm)) := \left\{f\in L^2(\R^2) \;\bigg\lvert\; \exists a\in \ell^2(\Z^2), \;f(x) = \sum_{\bk\in \Z^2} a(\bk) B_{\balpha}(\bx - \bk \vb \bXm)\right\}.
\eeq
Note that for $|\balpha|  > 1$, the space $\cS^{\balpha}$ is a principal shift-invariant subspace of $L^2(\R^2)$. Our goal is to derive conditions on the knots such that the family $\{B_{\balpha}(\bullet - \bk \vb \bXm) \vb \bk\in \Z^2\}$ forms a Riesz basis for $\cS^{\balpha}$.

We proceed as follows. For $B_{\balpha}(\bullet \vb \bXm)$, we consider the $(2\pi,2\pi)$-periodic function $[\widehat{B_{\balpha}}](\bullet \vb \bXm)$ defined by
\beq\label{perB}
[\widehat{B_{\balpha}}](\bomega \large\vb \bXm) := \sum_{\bk\in\Z^2 } \lvert\widehat{B_{\balpha}}(\bomega + 2\pi \bk \vb \bXm)\rvert^2.
\eeq
We employ \cite[Proposition 5.7 (i)]{book:woj} to show that for $|\balpha| > 1$ the family $\{B_{\balpha}(\bullet - \bk \vb \bXm) \vb \bk\in \Z^2\}$ is a Riesz sequence in $L^2(\R^2)$, i.e., that there exist two positive constants $c$ and $C$ such that
\beq\label{rieszbounds}
0 < c \leq [\widehat{B_{\balpha}}](\bomega \vb \bXm) \leq C < \infty,
\eeq
for almost all $\bomega\in \R^2$.

For this purpose, we recall the definition of $\widehat{B_{\balpha}}(\bomega \vb \bXm)$, namely,
\[
\widehat{B_{\balpha}}(\bomega \vb \bXm) = \left(\frac{1 - e^{-i\langle \bomega, \bx^1\rangle}}{i \inn{\bomega}{\bx^1}}\right)^{\alpha_1}\left(\frac{1 - e^{-i\langle \bomega, \bx^2\rangle}}{i \inn{\bomega}{\bx^2}}\right)^{\alpha_2}  \left(\frac{e^{i\langle \bomega, \bx^3\rangle} - 1}{i\inn{\bomega}{\bx^3}}\right)^{\alpha_3},\quad\balpha\in \R_+^3.
\]
Therefore,
\beq\label{sincs}
 \lvert \widehat{B_{\balpha}}(\bomega + 2\pi \bk \vb \bXm)\rvert^2 = \prod_{j=1}^3\left|\sinc\frac{\inn{\bomega + 2\pi \bk}{\bx^j}}{2}\right|^{2\alpha_j}.
\eeq
Notice that the function $[\widehat{B_{\balpha}}](\bullet \vb \bXm)$ is symmetric about $\bomega = (\pi,\pi)$ so that, together with its $(2\pi,2\pi)$-periodicity, is suffices to establish \eqref{rieszbounds} for almost all $\bomega\in [0,\pi]^2$.

To find a positive lower bound, we consider the summand $\bk = (0,0)$ in \eqref{perB}. As the $\sinc$ function is symmetric about $\inn{\bomega}{\bx^j}=0$ it suffices to consider only positive arguments. Now,
\[
\frac{\inn{\bomega}{\bx^j}}{2} \leq \frac{\pi ( |\bx^j_1| + |\bx^j_2|)}{2}.
\]
If the direction $\bx^j$ is chosen in such a way that there exists a positive constant $\vartheta_j$ so that
\beq\label{xcon}
|x^j_1| + |x^j_2| \leq \vartheta_j < 2,
\eeq
then, as the $\sinc$ is decreasing over $0 < \inn{\bomega}{\bx^j} < \pi$,
\[
\left|\sinc\frac{\inn{\bomega}{\bx^j}}{2}\right| \geq \frac{2}{\pi\vartheta_j},\quad \forall\,\bomega\in [0,\pi]^2.
\]
Hence, we may choose $c:= \displaystyle{\prod_{j=1}^3\left(\frac{2}{\pi\vartheta_j}\right)^{2\alpha_j}} > 0$. Note that since $\bx^3 = (1,0)^T$, condition \eqref{xcon} is automatically fulfilled and one may choose $\vartheta_3 = 1$.

To obtain an upper bound, we proceed as follows. Suppose that in addition to \eqref{xcon}, the directions $\bx^j = (x^j_1,x^j_2)^T$ satisfy
\beq\label{tau}
\tau_j := \frac{x_2^j}{x_1^j} \notin \Q, \quad j = 1,2,
\eeq
where we assumed without loss of generality that $x_1^j \neq 0$. We remark on the case $j = 3$ below.

Note that since $\bomega$ is bounded and $\bx^j$ satisfies conditions \eqref{xcon} and \eqref{tau}, there exists a $\bk_0 := (k_{0,1},k_{0,2})\in \Z^2$ so that $|\inn{\bomega + 2\pi \bk}{\bx^j}| = |\inn{\bomega}{\bx^j} + 2\pi \inn{\bk}{\bx^j}|> 0$, for all $\bk = (k_1,k_2)$ with $|k_i|\geq |k_{0,i}|$, $i = 1, 2$, and all $j = 1,2$. Let us denote the collection of all such $\bk$ by $K$. For a $\bk\in K$, the function $\bomega \mapsto |\inn{\bomega + 2\pi \bk}{\bx^j}|$ attains then its unique minimum value, which is strictly positive, at one of the endpoints of the square $[0,\pi]^2$. Denote this minimum value by $\pi\inn{\bm}{\bx^j}$. The collection of all such $\bm = \bm(\bk)$ constitutes a subset $M$ of $\Z^2$.

Thus
\begin{align*}
[\widehat{B_{\balpha}}](\bomega \vb \bXm) &= \sum_{\bk\in\Z^2 } \prod_{j=1}^3\left|\sinc\frac{\inn{\bomega + 2\pi \bk}{\bx^j}}{2}\right|^{2\alpha_j}\\
& = \sum_{\bk\in\Z^2\setminus K } \prod_{j=1}^3\left|\sinc\frac{\inn{\bomega + 2\pi \bk}{\bx^j}}{2}\right|^{2\alpha_j} + \sum_{\bk\in K} \prod_{j=1}^3\left|\sinc\frac{\inn{\bomega + 2\pi \bk}{\bx^j}}{2}\right|^{2\alpha_j}\\
& \leq \card (\Z^2\setminus K) + \prod_{j=1}^3 \sum_{\bk\in K} \left|\sinc\frac{\inn{\bomega + 2\pi \bk}{\bx^j}}{2}\right|^{2\alpha_j}.
\end{align*}

Now,
\[
\left| \sinc \frac{\inn{\bomega + 2\pi \bk}{\bx^j}}{2}\right| \leq \begin{cases}
\displaystyle{\frac{2}{\pi |\inn{\bm}{\bx^j}|}}, & j =1,2;\\
\displaystyle{\frac{2}{\pi (2|k_1| - 1)}}, & j =3,
\end{cases}
\quad \forall \bomega\in [0,\pi]^2,\;\forall \bk\in K,
\]
and, therefore, for $j = 1,2$,
\[
\sum_{\bk\in K} \left|\sinc\frac{\inn{\bomega + 2\pi \bk}{\bx^j}}{2}\right|^{2\alpha_j} \leq\sum_{\bm\in M}\left(\frac{2}{\pi |\inn{\bm}{\bx^j}|}\right)^{2\alpha_j} \leq \left(\frac{2}{\pi}\right)^{2\alpha_j} \sum_{\bm\neq (0,0)}\frac{1}{|\inn{\bm}{\bx^j}|^{2\alpha_j}}.
\]
Rewriting the last sum, produces
\begin{align*}
\sum_{\bm\neq (0,0)}\frac{1}{|\inn{\bm}{\bx^j}|^{2\alpha_j}} &= \sum_{(m_1,m_2)\neq (0,0)}\frac{1}{|x_1^j|^{2\alpha_j}|m_1 + \tau_j m_2|^{2\alpha_j}} = \zeta (Q_j ; \alpha_j),\quad \alpha_j > 1.
\end{align*}
Here, $ \zeta (Q_j ; \alpha_j)$ denotes the Epstein zeta function for the positive definite quadratic form $Q_j (m_1,m_2) := (x_1^j)^2 |m_1 + \tau_j m_2|^2$. The Epstein zeta function converges for $\alpha_j > 1$ \cite{article:zw}.

In the case $j = 3$, we obtain instead
\begin{align*}
\sum_{\bk\in K} \left|\sinc\frac{\inn{\bomega + 2\pi \bk}{\bx^3}}{2}\right|^{2\alpha_3} &\leq \left(\frac{2}{\pi}\right)^{2\alpha_3} \sum_{k_1\in \Z\setminus\{0\}} \frac{1}{(2|k_1|-1)^{2\alpha_3}} \leq \left(\frac{2}{\pi}\right)^{2\alpha_3} \sum_{k_1\in \N} \frac{1}{k_1^{2\alpha_3}}\\
& = \zeta(2\alpha_3),
\end{align*}
where $\zeta(2\alpha_3)$ is the Riemann zeta function, which converges for $2\alpha_3 >1.$

Thus, we may select for the positive constant $C$ as an upper bound in \eqref{rieszbounds} the value
\[
C := \card (\Z^2\setminus K) + \left(\frac{2}{\pi}\right)^{2|\balpha|} \zeta (Q_1 ; \alpha_1) \zeta (Q_2 ; \alpha_2) \zeta (2\alpha_3) < \infty,
\]
with $\alpha_j > 1$, $j = 1,2$, and $\alpha_3 > \frac12$.

We summarize the above findings in the next theorem.

\begin{rieszbasis}\label{theorem:riesz}
Let $\bXm = \left\{ \bx^1, \bx^2, -\bx^3\right\}\subset \mathbb{R}^2$, where $\bx^3 = \bx^1 + \bx^2$, be a knot set in $\R^2$. Suppose that the knots $\bx^1$ and $\bx^2$ satisfy the following two conditions:
\begin{enumerate}
\item[\emph{(i)}] $|x^j_1| + |x^j_2| \leq \vartheta_j < 2$, $\quad j = 1,2$.
\item[\emph{(ii)}] $\displaystyle{\frac{x_2^j}{x_1^j}} \notin \Q, \quad j = 1,2$. (Assuming without loss of generality that $x_1^j \neq 0$.)
\end{enumerate}
Further assume that $|\balpha| > 1$. Then the family $\{B_{\balpha}(\bullet - \bk \vb \bXm) \vb \bk\in \Z^2\}$ of bivariate hex splines is a Riesz basis for
the principal shift-invariant subspace $\cS^{\balpha}\subset L^2(\R^2)$.
\end{rieszbasis}

It was our hope to show that in connection with the refinement equation \eqref{twoscale} and under the hypotheses of Theorem \ref{theorem:riesz} that the spaces $V^{\balpha}_\ell$ defined by
\[
V^{\balpha}_\ell := \overline{\spn \left\{B_{\balpha}(2^\ell \bullet - M(\bk-\balpha_3) \vb \bXm) \vb \bk\in\Z^2\right\}}^{L^2(\R^2)}, \quad \ell \in \Z,
\]
form a dyadic multiresolution analysis of $L^2(\R^2)$. The properties (a) $\displaystyle{\bigcup_{\ell\in \Z} V_\ell^{\balpha}}$ is dense in $L^2(\R^2)$ and (b) $\displaystyle{\bigcap_{\ell\in \Z} V_\ell^{\balpha}} = \emptyset$ follow immediately from \cite[Theorems 2.3.2 and 2.3.4]{book:nov}. Arguments similar to those employed above lead to a lower Riesz bound $c > 0$, but it is not possible to attain a finite upper Riesz bound $C$.
\subsection{Extension to Complex Orders}
Definitions \ref{fracconedefn} and \ref{frachexdefn} can be extended to complex orders $z$. As the previous section shows, fractional cone and hex splines provide continuous (with respect to order of smoothness) families of functions. A complex order generates a family of complex-valued functions which also contain phase information as will be seen below.

We restate Definitions \ref{fracconedefn} and \ref{frachexdefn} for complex orders.

\begin{complexconedefn}
Let $\bX = \left\{ \bx^1, \ldots, \bx^n\right\} \subset \mathbb{R}^n \backslash \left\{ \mathbf{0}\right\}$, $n \geq s$, with $\spn(\bX) =\mathbb{R}^s$ and assume $\bz := \left( z_1,\ldots, z_n\right) \in \mathbb{C}_+^n := \left\{ \bz \in \mathbb{C}^n \vb \re z_k > 0\right\}$.  We define the \emph{cone spline of complex order $\bz$} in terms of its Fourier transform as the tempered distribution given by
\begin{equation}\label{compcone}
    \widehat{C_\bz}(\bomega \vb \bX) =  \prod_{k=1}^n \left( \frac{1}{i\langle \bomega, \bx^k\rangle} + \sigma_k \pi \delta\left(\langle \bomega, \bx^k\rangle\right)\right)^{z_k}
\end{equation}
where
\begin{equation*}
    \sigma_k = \left\{ \begin{array}{ll} 0,& z_k \notin \Z_+ \\ 1, & z_k \in \Z_+\end{array}\right. .
\end{equation*}
\end{complexconedefn}
Similarly, we define a hex spline of complex order.
\begin{complexhexdefn}
Let $\bX = \left\{ \bx^1, \bx^2, \bx^3 \right\}$ be a $3$-directional mesh in $\mathbb{R}^2$ with $\bx^1, \bx^2$ linearly independent and $\bx^3 = \bx^1 + \bx^2 = (1,0)^T$ and assume $\bz \in \mathbb{C}_+^3$.  We define the \emph{hex spline of complex order $\bz$} in terms of its Fourier transform as follows:
\begin{align}\label{comphex}
    \widehat{B_\bz}(\bomega \vb \bXm) &:=  \left( 1 - e^{-i\langle \bomega, \bx^1\rangle}\right)^{z_1}\left( 1 - e^{-i\langle \bomega, \bx^2\rangle}\right)^{z_2}  \left( e^{i\langle \bomega, \bx^3\rangle} - 1\right)^{z_3} \widehat{C_\bz}(\bomega \vb \bX)\nonumber\\
    & = \left(\frac{1 - e^{-i\langle \bomega, \bx^1\rangle}}{i \inn{\bomega}{\bx^1}}\right)^{z_1}\left(\frac{1 - e^{-i\langle \bomega, \bx^2\rangle}}{i \inn{\bomega}{\bx^2}}\right)^{z_2}  \left(\frac{e^{i\langle \bomega, \bx^3\rangle} - 1}{i\inn{\bomega}{\bx^3}}\right)^{z_3},
\end{align}
where $\bXm = \left\{ \bx^1, \bx^2, -\bx^3\right\}$.
\end{complexhexdefn}
Recall the fact that for complex numbers $z$ and $w$, $w^z := e^{z \log w} = e^{z (\log |w| + i \arg w)}$ is uniquely defined for $\arg w \in [-\pi, \pi)$. Hence, we consider only the principal branch of the complex exponential function. As usual, we set $0^z := 0$ and $w^0 = 1$.

Define
\[
\Theta_k := \Theta (\bomega \vb \bx^k) := \frac{1}{i\langle \bomega, \bx^k\rangle} + \sigma_k \pi \delta\left(\langle \bomega, \bx^k\rangle\right)
\]
and
\[
\Omega_k := \Omega (\bomega \vb \bx^k) := \frac{1 - e^{-i\langle \bomega, \bx^k\rangle}}{i \inn{\bomega}{\bx^k}},\quad\text{(obvious modification for $k = 3$)}.
\]
Note that $\Omega_k$ is well-defined as it never crosses the negative real axis. For $\im \Omega_k = \frac{\cos\inn{\bomega}{\bx^k}-1}{\inn{\bomega}{\bx^k}} = 0$ iff $\inn{\bomega}{\bx^k}\in 2\pi \Z$ and $\re\Omega_k = \frac{\sin\inn{\bomega}{\bx^k}}{\inn{\bomega}{\bx^k}}$ equals 1, for $\inn{\bomega}{\bx^k} = 0$, and 0 for $\inn{\bomega}{\bx^k} \in 2\pi\Z\setminus\{0\}$.

With these observations, Equations \eqref{compcone} and \eqref{comphex} can be rewritten in the following form.
\beq\label{44}
\widehat{C_\bz}(\bomega \vb \bX) = \widehat{C_{\re \bz}}(\bomega \vb \bX)\,\exp\left(-\sum_k \im z_k \cdot \arg \Theta_k\right)\,\exp\left(i\,\sum_k \im z_k\cdot \log |\Theta_k|\right)
\eeq
and
\beq\label{45}
\widehat{B_\bz}(\bomega \vb \bXm) = \widehat{B_{\re \bz}}(\bomega \vb \bXm)\,\exp\left(-\sum_k \im z_k \cdot \arg \Omega_k\right)\,\exp\left(i\,\sum_k \im z_k\cdot \log |\Omega_k|\right),
\eeq
where we set $\re \bz := (\re z_1, \ldots, \re z_n)$.

Equations \eqref{44} and \eqref{45} show that both the cone spline and hex spline consists of a cone spline, respectively, hex spline of fractional order $\re z$, multiplied by a modulation factor $\exp\left(-\displaystyle{\sum_k}\, \im z_k \cdot \arg \Theta_k\right)$ and a phase factor $\exp\left(i\,\displaystyle{\sum_k}\, \im z_k\cdot \log |\Omega_k|\right)$.  A complex cone spline is plotted in Figure \ref{fig:complexcones2}.

\begin{figure}[ht]
    \begin{center}
        \resizebox{\textwidth}{!}{\includegraphics{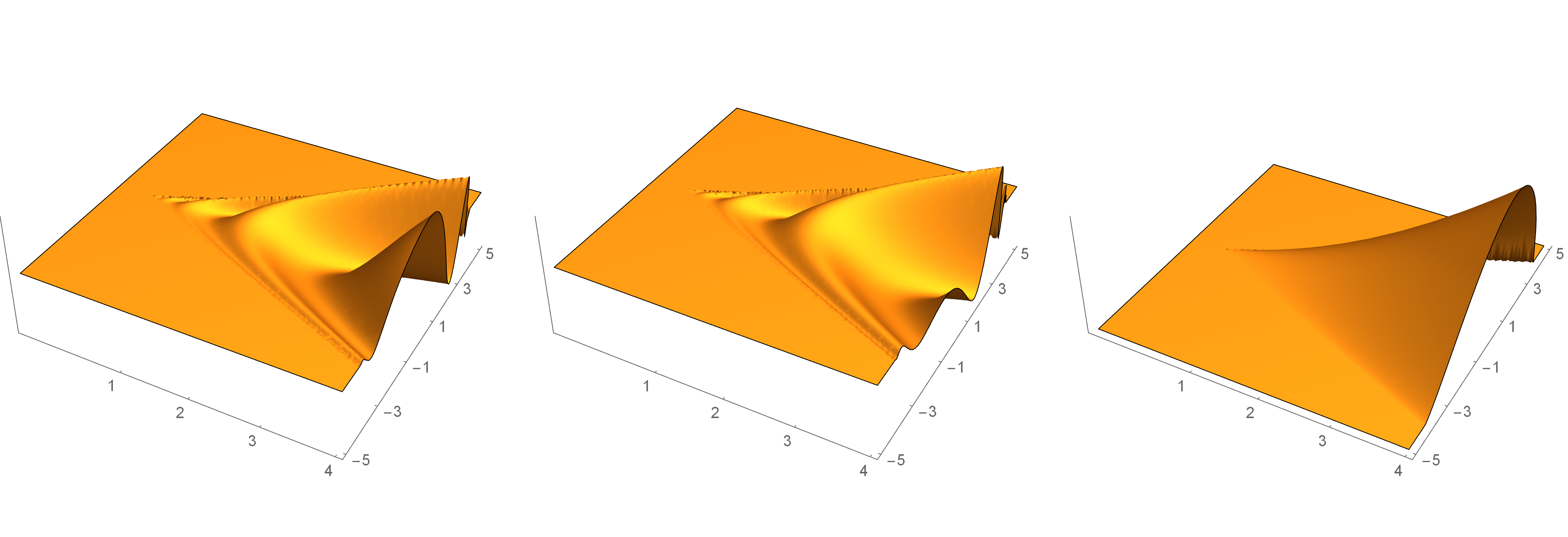}}
    \end{center}
    \caption{The complex cone spline $C_\bz(\bomega \vb \bX)$ with knot set $\bX = \left\{ (1,1)^T, (1,-1)^T\right\}$ and $\mathbf{z} = \left( 2.25 + 2.75i, 1.5 + 3.5i\right)$.  Plotted left to right are $\re C_\bz(\bomega \vb \bX)$, $\im C_\bz(\bomega \vb \bX)$, and $\vb C_\bz(\bomega \vb \bX) \vb$.}
    \label{fig:complexcones2}
\end{figure}

The existence of these two factors may allow the extraction of additional information from sampled data and the manipulation of images. In fact, the spectrum of a complex hex spline consists of the spectrum of a fractional hex spline combined with a modulating and a damping factor. The presence of an imaginary part causes the frequency components along the negative and positive direction of $\bx^k$ to be enhanced with different signs. This has the effect of shifting the frequency spectrum towards the negative or positive frequency side along $\bx^k$, depending on the sign of the imaginary part.

Moreover, in contrast to, for instance, complex wavelet bases, the phase information $\exp\left(i\,\displaystyle{\sum_k}\, \im z_k\cdot \log |\Omega_k|\right)$ is already built in, and an adjustable smoothness parameter, namely $\re\bz$, provides a continuous family of functions. The importance of complex-valued transforms in image analysis is discussed in \cite{article:forster}.

The potential applicability of hex splines to image analysis based on the above observations will be investigated in a forthcoming paper and published elsewhere.

The results in the previous subsections regarding the time domain representation of cone and hex splines, the decay rate and refinement property as well as the Riesz basis property easily extend to complex orders; $\balpha$ needs to be replaced by $\re \bz$ and $\alpha_k$ by $\re z_k$.

    \section*{Acknowledgement}
        The first author is partially supported by DFG grant MA5801/2-1.

    \bibliography{fractionalsplines}

    \vspace{.1in}

    \noindent\textbf{Peter Massopust\\
    Zentrum Mathematik, M6\\
    Technische Universit\"at M\"unchen\\
    Boltzmannstr. 3\\
    85747 Garching, Germany\\
    massopust@ma.tum.de}

    \vspace{.1in}\noindent 
    \textbf{Patrick Van Fleet\\
    Department of Mathematics\\
    University of St. Thomas\\
    2115 Summit Avenue\\
    Saint Paul, MN 55105, U.S.A.\\
    pjvanfleet@stthomas.edu
    }
\end{document}